\documentclass[12pt]{amsart}

\usepackage{amsmath,amssymb,amsthm,enumerate}
\usepackage[margin=1in]{geometry}

\numberwithin{equation}{section}

\DeclareMathOperator{\rad}{rad}

\DeclareMathOperator{\Aut}{Aut}
\DeclareMathOperator{\SL}{SL}

\newcommand{\N}{\mathbb{N}}
\newcommand{\Z}{\mathbb{Z}}

\newcommand{\R}{\mathbb{R}}

\newcommand{\ve}{\varepsilon}
\newcommand{\vp}{\varphi}

\newcommand{\1}{^{-1}}
\newcommand{\n}{^{n}}
\newcommand{\f}[2]{\frac{#1}{#2}}

\newcommand{\Mod}[1]{\ (\mathrm{mod}\ #1)}

\newtheorem{theorem}{Theorem}[section]

\newtheorem{definition}[theorem]{Definition}
\newtheorem{conjecture}[theorem]{Conjecture}
\newtheorem{corollary}[theorem]{Corollary}
\newtheorem{proposition}[theorem]{Proposition}
\newtheorem{lemma}[theorem]{Lemma}

\numberwithin{theorem}{section}

\title[An asymptotic version of the prime power conjecture]{An asymptotic version of the prime power conjecture for perfect difference sets}
\author{Sarah Peluse}
\address{Mathematical Institute, University of Oxford, Radcliffe Observatory Quarter, Woodstock Road, Oxford OX2 6GG, United Kingdom}
\email{sarah.peluse@maths.ox.ac.uk}

\begin{document}
\begin{abstract}
We show that the number of positive integers $n\leq N$ such that $\Z/(n^2+n+1)\Z$ contains a perfect difference set is asymptotically $\f{N}{\log{N}}$.
\end{abstract}

\maketitle

\section{Introduction}\label{sec1}

A subset $D\subset \Z/m\Z$ is a \textit{perfect difference set} if every nonzero $a\in\Z/m\Z$ can be written uniquely as the difference of two elements of $D$. For example, $\{1,2,4\}\subset\Z/7\Z$ is a perfect difference set. By a simple counting argument, if $D\subset\Z/m\Z$ is a perfect difference set, then we must have $m=n^2+n+1$ and $|D|=n+1$ for some integer $n$. In this situation, we say that the perfect difference set $D$ has \textit{order} $n$. Aside from being large Sidon sets, so that their existence and construction is of interest in additive number theory, perfect difference sets are also important objects of study in design theory and finite geometry (see the detailed account in~\cite{JungnickelPott1999}). Indeed, any perfect difference set $D$ of order $n$ gives rise to a finite projective plane of order $n$ by taking the set of points to be $\Z/(n^2+n+1)\Z$ and the set of lines to be translates of $D$. 

Singer~\cite{Singer1938} constructed perfect difference sets of every prime power order, and it is an old conjecture that these are the only orders for which perfect difference sets exist (see, for example,~\cite{Hall1947},~\cite{EvansMann1951}, or~\cite[C10]{Guy1994}). This conjecture is now referred to in the literature as the ``prime power conjecture'', and has been verified computationally for all $n$ up to $2$ billion by Baumert and Gordon~\cite{BaumertGordon2004}.
\begin{conjecture}[The prime power conjecture]\label{conj1.1}
An integer $n\geq 2$ is the order of a perfect difference set if and only if $n$ is a prime power.
\end{conjecture}

There are many partial results towards the prime power conjecture, though the conjecture itself seems out of reach. Some of the more general results say that all or almost all of the integers in certain congruence classes cannot be the order of a perfect difference set. For example, Bruck and Ryser~\cite{BruckRyser1949} showed that if $n$ is the order of a projective plane and $n\equiv 1,2\Mod{4}$, then $n$ can be written as the sum of two squares, Jungnickel and Vedder~\cite{JungnickelVedder1984} showed that if $n$ is the order of a perfect difference set and $2\mid n$, then $n=2$, $n=4$, or $8\mid n$, and Willbrink~\cite{Wilbrink1985} showed that if $n$ is the order of a perfect difference set and $3\mid n$, then $n=3$ or $9\mid n$. There are apparently no results saying that the set of orders of perfect difference sets has density zero in the integers, however.

In this paper, we prove that the set of orders of perfect difference sets has the asymptotic size predicted by the prime power conjecture.
\begin{theorem}\label{thm1.2}
  We have
  \[
\#\{n\leq N:\Z/(n^2+n+1)\Z\text{ contains a perfect difference set}\}=(1+o(1))\f{N}{\log{N}}.
  \]
\end{theorem}
This gives further evidence for the truth of the prime power conjecture, and implies that if counterexamples exist, they must be sparser than the primes. The proof of Theorem~\ref{thm1.2} gives the explicit expression $O(\exp(-C\f{\log\log\log\log\log{N}}{\log\log\log\log\log\log{N}}))$ for the $o(1)$ term above, though we made no serious attempt to optimize this bound. 

To prove Theorem~\ref{thm1.2}, we begin by splitting the set of $n\leq N$ up into various subsets depending on the prime factorization of $n^2+n+1$. To each of these sets, we apply one of two results from the theory of perfect difference sets. Both say that if $n$ is the order of a perfect difference set, then certain relations between the prime factors of $n$ and the prime factors of $n^2+n+1$ must hold. Applying these results thus turns the problem of proving Theorem~\ref{thm1.2} into that of bounding the size of sets defined by various number-theoretic conditions.

The remainder of this paper is organized as follows. In Section~\ref{sec2}, we state the results on perfect difference sets used in the proof of Theorem~\ref{thm1.2} and, in Section 3, give an outline of the argument. We count the number of non-prime-power orders $n\leq N$ of perfect difference sets such that $n^2+n+1$ has at least three, exactly two, and exactly one prime factor(s) in Sections~\ref{sec3},~\ref{sec4}, and~\ref{sec5}, respectively. The arguments in Sections~\ref{sec4} and~\ref{sec5} depend on estimates for the number of lattice points satisfying various size and congruence restrictions on certain hyperboloids. We delay the proofs of these lattice point counting results to Sections~\ref{sec6} and~\ref{sec7}.

\section*{Acknowledgments}
The author thanks Ben Green and Kannan Soundararajan for helpful conversations and comments on earlier drafts of this paper and the anonymous referee for many useful suggestions. The author is supported by the NSF Mathematical Sciences Postdoctoral Research Fellowship Program under Grant No. DMS-1903038

\section{Notation and preliminaries}\label{sec2}

We will first set some notation. For each $k\in\N$, let $\log_k$ denote the $k$-fold iterated logarithm, so that, for example, $\log_3x=\log\log\log{x}$. No logarithms to any base other than $e$ appear in this paper, so confusion should not arise. If $D\subset\Z/m\Z$ and $t,a\in\Z/m\Z$, we define the sets $t\cdot D$ and $a+D$ to be $\{td:d\in D\}$ and $\{a+d:d\in D\}$, respectively. Throughout this paper, $p$ and $q$ will always denote prime numbers. For any Dirichlet character $\chi$ and $y>0$, let $L(1,\chi;y)$ denote the Euler product $\prod_{p<y}(1-\chi(p)/p)\1$. For any $a\in\Z/p\Z$, we will use $\delta_{a}$ to denote the function that is $1$ at $a$ and $0$ otherwise. For every prime $q>2$, set $q^*:=(-1)^{\f{q-1}{2}}q$. For every $n,k\in\N$, we let $p_k(n)$ denote the $k^{th}$ smallest prime factor of $n$ with the convention that $p_k(n)=\infty$ if $\omega(n)<k$, so that $p_1(n)<\dots<p_k(n)$ whenever $\omega(n)\geq k$. Letting $\mathcal{P}$ denote the set of prime powers, we set
\[
S(N):=\{n\leq N:\Z/(n^2+n+1)\Z\text{ contains a perfect difference set}\}\setminus\mathcal{P},
\]
the set of non-prime-power orders of perfect difference sets in $\{1,\dots,N\}$. By Singer's construction, to prove Theorem~\ref{thm1.2} it suffices to show that $\#S(N)=o(\f{N}{\log{N}})$.

We now state the two results from the theory of perfect difference sets used in this paper. The first is due to Mann~\cite{Mann1952}.
\begin{theorem}[Mann,~\cite{Mann1952}]\label{thm2.1}
Let $n$ be the order of a perfect difference set, and assume that $n$ is not a perfect square. If $p$ and $q$ are primes such that $p\mid n$ and $q\mid n^2+n+1$, then $p$ is a quadratic residue modulo $q$.
\end{theorem}

Note that the condition imposed by Mann's theorem is empty when $n^2+n+1$ is a prime congruent to $1$ modulo $4$, which we expect to happen for $\gg\f{N}{\log{N}}$ of the $n\leq N$ by the Bateman--Horn conjecture. Indeed, since $n^2+n+1\equiv 1\Mod{n}$, quadratic reciprocity tells us that every odd prime dividing $n$ is a square modulo $n^2+n+1$ whenever $n^2+n+1$ is a prime congruent to $1$ modulo $4$. The contribution of such $n$ must be dealt with if we want to prove Theorem~\ref{thm1.2}, and not just a weaker big-$O$ result. To do so, we will use the following lemma.

\begin{lemma}\label{lem2.2}
Let $n$ be the order of a perfect difference set, and assume that $q:=n^2+n+1$ is prime and $q\equiv 1\Mod{4}$. If $p$ is a prime such that $p\mid n$, then $p$ is a quartic residue modulo $q$.
\end{lemma}
To prove Lemma~\ref{lem2.2}, we will need some basic facts from the theory of \textit{multipliers} of perfect difference sets.

\begin{definition}
Let $D$ be a perfect difference set of order $n$. We say that $t\in(\Z/(n^2+n+1)\Z)^\times$ is a \textit{numerical multiplier} for $D$ if $t\cdot D=a+D$ for some $a\in\Z/(n^2+n+1)\Z$.
\end{definition}
Note that the set of numerical multipliers of a perfect difference set is closed under multiplication. Mann showed that every perfect difference set has a translate that is fixed by all of its numerical multipliers (this result is attributed to Mann by Hall in~\cite{Hall1947}), and Hall~\cite{Hall1947} showed that if $D$ is a perfect difference set of order $n$, then every prime dividing $n$ is a numerical multiplier of $D$.
\begin{proof}[Proof of Lemma~\ref{lem2.2}]
We may assume, without loss of generality, that $D$ is fixed under multiplication by any of its numerical multipliers. Note that $-1$ cannot be a numerical multiplier of $D$. Indeed, we must have $|D|\geq 3$, so that there exist distinct $d,d'\in D$ such that $d\neq -d'$. Observe, however, that $d-d'=(-d')-(-d)$, so that no such perfect difference set $D$ can satisfy $D=-D$.

By Hall's result, we have that  $p^i\not\equiv -1\Mod{q}$ for any $i\geq 0$, so that $p$ must have odd multiplicative order modulo $q$. Since $4\mid q-1$, this implies that  $p$ must be a quartic residue modulo $q$.
\end{proof}

\section{Outline of the proof of Theorem~\ref{thm1.2}}

Given Theorem~\ref{thm2.1} and Lemma~\ref{lem2.2}, it should not be surprising that $\# S(N)=o(\f{N}{\log{N}})$. Indeed, for a typical integer $n$, one of $n$ or $n^2+n+1$ will have enough prime factors that the conditions imposed by these results should be very rarely satisfied. The difficulty with turning this heuristic into a proof is that $n^2+n+1$ (and thus its prime factors) obviously depends on $n$. The conditions in Theorem~\ref{thm2.1} and Lemma~\ref{lem2.2} are sufficiently powerful, however, that we can afford to use the union bound in several places, which allows us to remove the dependence of the prime factors of $n^2+n+1$ on a few of the small prime factors of $n$. To do this effectively, we must use different techniques depending on the prime factorization of $n^2+n+1$.

The contribution to $\# S(N)$ coming from $n$ such that $n^2+n+1$ has at least three prime factors is the most straightforward to handle. Typically, such $n$ are divisible by two distinct primes $p_1$ and $p_2$ satisfying $3\neq p_1,p_2\leq N^{\f{1}{1000}}$, say (the number of such $n\in S(N)$ not satisfying this condition can be shown to be $\ll\f{N}{(\log{N})^{3/2}}(\log_2{N})^{O(1)}$ using an argument similar to the one about to be sketched). Thus, since $n^2+n+1$ must have a prime divisor below $N^{\f{2}{3}}$ in this situation, by using the union bound and Theorem~\ref{thm2.1} it suffices to bound
\begin{equation}\label{eq2.1}
\sum_{\substack{ q\leq N^{\f{2}{3}} \\ p_1<p_2\leq N^{\f{1}{1000}} \\ 3\neq p_1,p_2}}\# S_{p_1,p_2}^q(N),
\end{equation}
where $S_{p_1,p_2}^q(N)$ equals
  \begin{align*}
\bigg\{n\leq N:p_1p_2\mid n,\ q\mid n^2+n+1,\text{ and } p'\mid n,\ q'\mid n^2+n+1\implies \left(\f{p'}{q}\right)=\left(\f{p_1}{q'}\right)=\left(\f{p_2}{q'}\right)=1\bigg\}.
  \end{align*}
  Bounding the size of each $S_{p_1,p_2}^q(N)$ is a sieve problem of dimension $\f{5}{4}$, the key being that $\f{5}{4}>1$. Thus, by an application of an upper bound sieve, we have that~\eqref{eq2.1} is bounded above by $\f{N}{(\log{N})^{5/4}}$ times a quantity of the form
\[
O\left(\sum_{\substack{ q\leq N^{\f{2}{3}} \\ p_1<p_2\leq N^{\f{1}{1000}} \\ 3\neq p_1,p_2}}\f{\left(\text{an Euler product that is typically small}\right)}{p_1p_2q}\right),
\]
which we can bound by a power of $\log_2{N}$.

To estimate the number of $n\in S(N)$ such that $n^2+n+1=q_1q_2$ for two primes $q_1<q_2$, we must split into subcases depending on the size of $q_1$. When $q_1\leq\f{N}{(\log{N})^{\beta}}$ for $\beta>0$ sufficiently large, an argument similar to the one above can be used. When $\f{N}{(\log{N})^{\beta}}\leq q_1\leq\f{N}{(\log{N})^{1/2}}$, a more delicate argument is required. To deal with this subcase, we split such $n$ up based on the smallest $k=k(N)\to\infty$ prime factors of $n$ below $\log_3{N}$ (the number of such $n$ without $k$ prime factors below $\log_3{N}$ is negligible), so that, by Theorem~\ref{thm2.1}, only asymptotically $2^{-k}$ times the number of primes $q$ in the interval $[\f{N}{(\log{N})^{\beta}},\f{N}{(\log{N})^{1/2}}]$ can possibly divide $n^2+n+1$. We then take the union bound over these $q$, apply an upper bound sieve, and sum over $q$ and the $k$-tuples of distinct primes below $\log_3{N}$. Finally, to deal with the subcase $\f{N}{(\log{N})^{1/2}}\leq q_1\leq N$, we forget the condition $n\in S(N)$ and show, using an enveloping sieve argument, that there are $\ll\f{N}{(\log{N})^{3/2}}$ many $n\leq N$ such that $n^2+n+1=q_1q_2$ with $q_1<q_2$ and $\f{N}{(\log{N})^{1/2}}\leq q_1\leq N$. One of the key inputs is an asymptotic count, with power saving error term, for the number of integer triples $(x,y,z)$ on the hyperboloid $y^2-4xz=-3$ satisfying $1\leq x,z\leq X$, $k\mid x$, and $\ell\mid z$, for a variety of $k$ and $\ell$. We prove an estimate for the number of these lattice points by adapting an argument of Hooley~\cite{Hooley1963}.

When $n^2+n+1$ is a prime, different arguments are required depending on whether $n^2+n+1$ is congruent to $1$ or $3$ modulo $4$. The number of $n\in S(N)$ such that $n^2+n+1$ is a prime congruent to $3$ modulo $4$ can be bounded easily using an upper bound sieve--it follows from Theorem~\ref{thm2.1} and quadratic reciprocity that if $n$ is not a perfect square, then every odd prime $p\mid n$ must satisfy $p\equiv 1\Mod{4}$. The situation when $n^2+n+1$ is congruent to $1$ modulo $4$ is much more involved. As in the second subcase of the paragraph above, we begin by splitting such $n$ up based on the smallest $k$ prime factors $p_1,\dots,p_k$ of $n$, but this time apply Lemma~\ref{lem2.2} to get that $p_1,\dots,p_k$ must all be quartic residues modulo $n^2+n+1$. By one of the formulations of the quartic reciprocity law, this forces $n^2+n+1$ to be representable by the quadratic form $x^2+4y^2$ with $y$ satisfying certain congruence conditions that depend on $p_1,\dots,p_k$. We bound the number of such $n$ by combining the Selberg sieve with an asymptotic count, with power-saving error term, for the number of integer triples $(x,y,z)$ on the hyperboloid $4x^2+16y^2-z^2=3$ satisfying $1\leq z\leq X$ and various congruence restrictions on $y$ and $z$. The proof of this lattice point counting result is also an adaptation of the previously mentioned argument of Hooley, though the argument ends up being significantly more complicated than the one for the other lattice point count.

\section{$n^2+n+1$ has at least three prime factors}\label{sec3}

In this section, we bound the number of $n\in S(N)$ such that $n^2+n+1$ has at least three prime factors:
\begin{proposition}\label{prop3.1}
  We have
  \[
\#\{n\in S(N):\Omega(n^2+n+1)\geq 3\}\ll\f{N}{(\log{N})^{\f{5}{4}}}(\log_2{N})^{3}.
  \]
\end{proposition}

We split the estimation of the number of $n\in S(N)$ such that $n^2+n+1$ has at least three prime factors into the estimation of the size of the following three sets:
\[
\{n\in S(N):3\nmid n,\ p_2(n)>N^{\alpha},\text{ and }\Omega(n^2+n+1)\geq 3\},
\]
\[
\{n\in S(N):3\mid n,\ p_3(n)>N^{\alpha},\text{ and }\Omega(n^2+n+1)\geq 3\},
\]
and
\[
\{n\in S(N):p_1p_2\mid n\text{ for some }p_1<p_2\leq N^{\alpha}\text{ with }p_1,p_2\neq 3\text{ and }\Omega(n^2+n+1)\geq 3\},
\]
for some $0<\alpha<\f{1}{6}$ to be fixed shortly. Note that if $n\in\N$ is not divisible by two distinct primes $p_1,p_2\leq N^\alpha$ with $p_1,p_2\neq 3$, then either $3\nmid n$ and the second smallest prime factor of $n$ has size at least $N^{\alpha}$, or $3\mid n$ and $n$ either has at most two prime factors or (since $3$ must then be either the smallest or second smallest prime factor of $n$) the third smallest prime factor of $n$ has size at least $N^{\alpha}$. Thus, to prove Proposition~\ref{prop3.1}, it really does suffice to bound the sizes of the above three sets. We begin by applying the union bound, Theorem~\ref{thm2.1}, and an upper bound sieve to deduce initial bounds for each.

\begin{lemma}\label{lem3.2}
  There exist absolute constants $0<\alpha<\f{1}{6}$ and $0<\gamma<1$ such that
  \[
\#\{n\in S(N):3\nmid n,\ p_2(n)>N^{\alpha},\text{ and }\Omega(n^2+n+1)\geq 3\}
  \]
  and
  \[
\#\{n\in S(N):3\mid n,\ p_3(n)>N^{\alpha},\text{ and }\Omega(n^2+n+1)\geq 3\}
  \]
  are both
  \[
\ll\f{N}{(\log{N})^{\f{3}{2}}}\sum_{3\neq p\leq N^{\f{1}{2}}}\f{L(1,\chi_{4^{\epsilon_p}p};N^{\gamma})^{\f{1}{2}}L(1,\chi_{-3\cdot 4^{\epsilon_{-3p}}p};N^{\gamma})^{\f{1}{2}}}{p},
  \]
  where $\epsilon_n=0$ if $n\equiv 1\Mod{4}$ and $\epsilon_n=1$ if $n\equiv 2,3\Mod{4}$, and
  \[
\#\{n\in S(N):p_1p_2\mid n\text{ for some }p_1<p_2\leq N^{\alpha}\text{ with }p_1,p_2\neq 3\text{ and }\Omega(n^2+n+1)\geq 3\}
  \]
is  
  \[
\ll\f{N}{(\log{N})^{\f{5}{4}}}\sum_{\substack{q\leq N^{\f{2}{3}}\\ p_1<p_2\leq N^{\alpha} \\ 3\neq p_1,p_2}}\f{\prod_{j=1}^{7}L(1,\chi_{m_j};N^{\gamma})^{k_j}}{p_1p_2q},
  \]
  where $m_1=q^*$, $m_2=4^{\epsilon_{p_1}}p_1$, $m_3=4^{\epsilon_{p_2}}p_2$, $m_4=4^{\epsilon_{p_1p_2}}p_1p_2$, $m_5=-3\cdot 4^{\epsilon_{-3p_1}}p_1$, $m_6=-3\cdot 4^{\epsilon_{-3p_2}}p_2$, $m_7=-3\cdot 4^{\epsilon_{-3p_1p_2}}p_1p_2$, $k_1=\f{1}{2}$, and $k_2=\dots=k_7=\f{1}{4}$.
\end{lemma}
\begin{proof}
We will apply a standard upper bound sieve, a statement of which can be found in Section~6.5 of~\cite{FriedlanderIwaniec2010}, numerous times throughout this paper, including multiple times within this proof.

By Theorem~\ref{thm2.1} (and an application of an upper bound sieve in the second inequality), we have
  \[
\#\{n\in S(N):3\nmid n,\ p_2(n)>N^{\alpha},\text{ and }\Omega(n^2+n+1)\geq 3\}\leq\sum_{3\neq p\leq N^{\f{1}{2}}}\# S_{1,p}(N)+O(N^{\f{1}{2}}),
  \]
  \[
\#\{n\in S(N):3\mid n,\ p_3(n)>N^{\alpha},\text{ and }\Omega(n^2+n+1)\geq 3\}\leq\sum_{3\neq p\leq N^{\f{1}{2}}}\# S_{2,p}(N)+O\left(\f{N}{(\log{N})^{\f{3}{2}}}+N^{\f{1}{2}}\right),
\]
and that
\[
\#\{n\in S(N):p_1p_2\mid n\text{ for some }p_1<p_2\leq N^{\alpha}\text{ with }p_1,p_2\neq 3\text{ and }\Omega(n^2+n+1)\geq 3\}
\]
is at most
  \[
\sum_{\substack{ q\leq N^{\f{2}{3}} \\ p_1<p_2\leq N^{\alpha} \\ 3\neq p_1,p_2}}\# S_{3,p_1,p_2}^q(N)+O(N^{\f{1}{2}}),
  \]
  where
  \[
S_{1,p}(N):=\left\{n\leq N:p\mid n,\ p_2(n)>N^\alpha,\text{ and }q\mid n^2+n+1\implies \left(\f{p}{q}\right)=1\right\},
  \]
  \[
S_{2,p}(N):=\left\{n\leq N:3p\mid n,\ p_3(n)>N^\alpha,\text{ and }q\mid n^2+n+1\implies \left(\f{p}{q}\right)=1\right\},
\]
and $S_{3,p_1,p_2}^q(N)$ equals
    \begin{align*}
\bigg\{n\leq N:p_1p_2\mid n,\ q\mid n^2+n+1,\text{ and } p'\mid n,\ q'\mid n^2+n+1\implies \left(\f{p'}{q}\right)=\left(\f{p_1}{q'}\right)=\left(\f{p_2}{q'}\right)=1\bigg\}.
  \end{align*}
The error term $O(N/(\log{N})^{3/2})$ appearing in the second inequality comes from the contribution of $n$ of the form $n=3^jp$ with $p\geq \sqrt{N}$, which we estimate using an upper bound sieve. The restriction that $n$ is only divisible by large primes and $3$ sieves out one congruence class modulo each prime on average (namely, the zero congruence class), and the restriction that $3\mid n$ forces $\left(\f{3}{q}\right)=1$ for all primes $q$ dividing $n^2+n+1$, which sieves out half of a (nonzero) congruence class modulo each prime on average (namely, the roots of $x^2+x+1$ modulo each prime where these roots exist and for which $3$ is a quadratic residue). This leads to the savings of $(\log{N})^{3/2}$ over the trivial bound. Similar arguments will appear numerous times throughout the remainder of this paper.

  Fixing $\alpha$ sufficiently small, to each of $S_{1,p}(N)$, $S_{2,p}(N)$, and $S_{3,p_1,p_2}^{q}(N)$ we apply an upper bound sieve to get that there exists a fixed constant $0<\gamma\leq\alpha$ such that
\[
\#S_{1,p}(N),\#S_{2,p}(N)\ll\f{N}{p}\prod_{p'<N^{\gamma}}\left(1-\f{g_1(p')}{p'}\right)
\]
for all $3\neq p'\leq N^{\f{1}{2}}$, where
\[
  g_1(p'):=1+\begin{cases} 2 & \left(\f{-3}{p'}\right)=1\text{ and }\left(\frac{p}{p'}\right)=-1 \\ 1 & p'=3\text{ and }\left(\f{p}{3}\right)=-1 \\ 0 & \text{otherwise} \end{cases}
\]
and such that
\[
\#S_{3,p_1,p_2}^q(N)\ll\f{N}{p_1p_2q}\prod_{p'<N^{\gamma}}\left(1-\f{g_3(p')}{p'}\right)
\]
for all $q\leq N^{\f{2}{3}}$ and $p_1<p_2\leq N^{\alpha}$ with $3\neq p_1,p_2$, where
\[
g_3(p'):=\begin{cases}1 & \left(\f{p'}{q}\right)=-1 \\ 0 & \text{otherwise} \end{cases}+\begin{cases} 2 & \left(\f{-3}{p'}\right)=1\text{ and }\left(\frac{p_1}{p'}\right)\text{ or }\left(\frac{p_2}{p'}\right)=-1 \\ 1 & p'=3\text{ and }\left(\f{p_1}{3}\right)\text{ or }\left(\f{p_2}{3}\right)=-1 \\ 0 & \text{otherwise} \end{cases}.
\]

Standard Euler product manipulations then yield
\[
\prod_{p'<N^{\gamma}}\left(1-\f{g_1(p')}{p'}\right)\ll\f{L(1,\chi_{4^{\epsilon_p}p};N^{\gamma})^{\f{1}{2}}L(1,\chi_{-3\cdot 4^{\epsilon_p}p};N^{\gamma})^{\f{1}{2}}}{(\log{N})^{\f{3}{2}}}
\]
when $3\neq p\leq N^{\f{1}{2}}$ and
\[
\prod_{p'<N^{\gamma}}\left(1-\f{g_3(p')}{p'}\right)\ll\f{\prod_{j=1}^{7}L(1,\chi_{m_j};N^{\gamma})^{k_j}}{(\log{N})^{\f{5}{4}}}
\]
when $p_1<p_2\leq N^\alpha$ with $3\neq p_1,p_2$ and $q\leq N^{\f{2}{3}}$,
\end{proof}
To finish the proof of Proposition~\ref{prop3.1}, we require a standard lemma (which will also be used once in Section~\ref{sec4}).
\begin{lemma}\label{lem3.3}
  Let $y>0$ and $a\in\Z$ be nonzero. We have
  \[
    \sum_{\substack{p\leq X \\ p\nmid a}}L(1,\chi_{ap};y),\sum_{\substack{p\leq X \\ p\nmid a}}L(1,\chi_{ap})\ll_a\f{X}{\log{X}}
  \]
  and
  \[
    \sum_{\substack{p_1<p_2\leq X \\ p_1,p_2\nmid a}}L(1,\chi_{ap_1p_2};y)\ll_a\f{X^2}{(\log{X})^2}.
  \]
\end{lemma}
This lemma follows from a small modification of the argument given in Section~5 of~\cite{Barban1966} by using a mean value estimate for sums of quadratic characters over primes due to Jutila~\cite{Jutila1981}. (Such a modification, in fact, gives asymptotics for the sums in Lemma~\ref{lem3.3}, and also for higher moments.) 

Now we can prove Proposition~\ref{prop3.1}.

\begin{proof}[Proof of Proposition~\ref{prop3.1}]
 By Lemma~\ref{lem3.2}, it suffices to bound
\begin{equation}\label{eq3.1}
\sum_{3\neq p\leq N^{\f{1}{2}}}\f{L(1,\chi_{4^{\epsilon_p}p};N^{\gamma})^{\f{1}{2}}L(1,\chi_{-3\cdot 4^{\epsilon_{-3p}}p};N^{\gamma})^{\f{1}{2}}}{p}
\end{equation}
and
\begin{equation}\label{eq3.2}
\sum_{\substack{q\leq N^{\f{2}{3}}\\ p_1<p_2\leq N^{\alpha} \\ 3\neq p_1,p_2}}\f{\prod_{j=1}^{7}L(1,\chi_{m_j};N^{\gamma})^{k_j}}{p_1p_2q},
\end{equation}
using the notation of Lemma~\ref{lem3.2}. By the Cauchy--Schwarz inequality, we have that~\eqref{eq3.1} is bounded above by
\[
\left(\sum_{3\neq p\leq N^{\f{1}{2}}}\f{L(1,\chi_{4^{\epsilon_p}p};N^{\gamma})}{p}\right)^{\f{1}{2}}\left(\sum_{3\neq p\leq N^{\f{1}{2}}}\f{L(1,\chi_{-3\cdot 4^{\epsilon_{-3p}}p};N^{\gamma})^{\f{1}{2}}}{p}\right)^{\f{1}{2}},
\]
which is $\ll\log_2{N}$ by Lemma~\ref{lem3.3} and partial summation. Similarly, by H\"older's inequality, we have that~\eqref{eq3.2} is bounded above by
\begin{align*}
  &\left(\sum_{q\leq N^{\f{2}{3}}}\f{L\left(1,\chi_{q^*};N^\gamma\right)^{\f{1}{2}}}{q}\right)\left(\sum_{p\leq N^{\alpha}}\f{L(1,\chi_{4^{\epsilon_p}p};N^\gamma)}{p}\right)^{\f{1}{2}}\left(\sum_{3\neq p\leq N^{\alpha}}\f{L(1,\chi_{-3\cdot 4^{\epsilon_{-3p}} p};N^\gamma)}{p}\right)^{\f{1}{2}}\\
  &\cdot\left(\sum_{\substack{p_1<p_2\leq N^{\alpha} \\ 3\neq p_1,p_2}}\f{L(1,\chi_{4^{\epsilon_{p_1p_2}}p_1p_2};N^\gamma)}{p_1p_2}\right)^{\f{1}{4}}\left(\sum_{\substack{p_1<p_2\leq N^{\alpha} \\ 3\neq p_1,p_2}}\f{L(1,\chi_{-3\cdot 4^{\epsilon_{-3p_1p_2}}p_1p_2};N^\gamma)}{p_1p_2}\right)^{\f{1}{4}},
\end{align*}
which is $\ll(\log_2{N})^3$, also by Lemma~\ref{lem3.3} and partial summation.
\end{proof}

\section{$n^2+n+1$ is the product of two primes}\label{sec4}

In this section, we bound the number of $n\in S(N)$ such that $n^2+n+1$ has exactly two prime factors:
\begin{proposition}\label{prop4.1}
  We have
  \[
\#\{n\in S(N):n^2+n+1\text{ is the product of two primes}\}\ll \f{N}{\log{N}\exp(C\f{\log_5{N}}{\log_6{N}})}
  \]
  for some absolute constant $C>0$.
\end{proposition}

Note that $n^2<n^2+n+1<(n+1)^2$ for all $n\geq 1$, so that $n^2+n+1$ is never the square of a prime. As outlined in Section~\ref{sec2}, we split the $n\in S(N)$ such that $n^2+n+1=q_1q_2$ with $q_1<q_2$ into three sets depending on the size of $q_1$:
\begin{equation}\label{eq4.1}
\left\{n\in S(N):n^2+n+1=q_1q_2\text{ with }q_1<q_2\text{ and }q_1\leq\f{N}{(\log{N})^{\beta}}\right\}
\end{equation}
\begin{equation}\label{eq4.2}
\left\{n\in S(N):n^2+n+1=q_1q_2\text{ with }q_1<q_2\text{ and }\f{N}{(\log{N})^{\beta}}<q_1\leq\f{N}{(\log{N})^{\f{1}{2}}}\right\}
\end{equation}
and
\begin{equation}\label{eq4.3}
\left\{n\in S(N):n^2+n+1=q_1q_2\text{ with }q_1<q_2\text{ and }\f{N}{(\log{N})^{\f{1}{2}}}<q_1\leq N\right\},
\end{equation}
for some $\beta>1$ to be fixed shortly. To prove Proposition~\ref{prop4.1}, it suffices to bound the size of each of the above three sets. Indeed, $n^2+n+1$ must have a prime factor of size at most $N$ whenever $n\leq N$ and $\Omega(n^2+n+1)= 2$, since $n^2+n+1<(n+1)^2$.

We begin by bounding the size of~\eqref{eq4.1}, using a modification of the argument presented in Section~\ref{sec3}.
\begin{lemma}\label{lem4.2}
There exists a $\beta>1$ such that
  \[
\#\left\{n\in S(N):n^2+n+1=q_1q_2\text{ with }q_1<q_2\text{ and }q_1\leq\f{N}{(\log{N})^{\beta}}\right\}\ll\f{N}{\log{N}(\log_2{N})^{\f{1}{2}}}.
  \]
\end{lemma}
\begin{proof}
  Arguing as in the proof of Lemma~\ref{lem3.2}, we first note that the left-hand side of the desired inequality is bounded above by
    \[
\sum_{q\leq \f{N}{(\log{N})^\beta}}\# S_{4}^q(N)+O(N^{\f{1}{2}}),
  \]
  where
  \[
S_4^q(N):=\left\{n\leq N:q\mid n^2+n+1\text{ and }p\mid n\implies \left(\f{p}{q}\right)=1\right\}.
\]
By an upper bound sieve, we have that
\[
  \#S_4^q(N)\ll\begin{cases}\f{N}{q\log(N/q)}\prod_{p'<(N/q)^\gamma}\left(1-\f{g_4(p')}{p'}\right) & N^{\f{2}{3}}<q\leq \f{N}{(\log{N})^\beta} \\
  \f{N}{q\log{N}}\prod_{p'<N^\gamma}\left(1-\f{g_4(p')}{p'}\right) & q\leq N^{\f{2}{3}} \\\end{cases},
\]
where
\[
g_4(p'):=\begin{cases} 1 & \left(\f{p'}{q}\right)=-1 \\ 0 & \text{otherwise}\end{cases}+\begin{cases} 2 & \left(\f{-3}{p'}\right)=1 \\ 1 & p'=3 \\ 0 &\text{ otherwise}\end{cases}
\]
and $\gamma>0$ is an absolute constant. Since
\[
\prod_{p'<y}\left(1-\f{g_4(p')}{p'}\right)\ll\f{L(1,\chi_{q^*};y)^{\f{1}{2}}}{(\log{y})^{\f{3}{2}}}
\]
for all $y>0$, it thus suffices to bound
\begin{equation}\label{eq4.4}
\f{1}{(\log{N})^{\f{3}{2}}}\sum_{q\leq N^{\f{2}{3}}}\f{L(1,\chi_{q^*};N^\gamma)^{\f{1}{2}}}{q}+\sum_{N^{\f{2}{3}}<q\leq \f{N}{(\log{N})^{\beta}}}\f{L(1,\chi_{q^*};(N/q)^\gamma)^{\f{1}{2}}}{q(\log(N/q))^{\f{3}{2}}}.
\end{equation}

That the first term of~\eqref{eq4.4} is $\ll\f{\log_2{N}}{(\log{N})^{3/2}}$ was already observed in the proof of Proposition~\ref{prop3.1}. To remove the dependence of the number of factors in the product $L(1,\chi_{q^*};(N/q)^\gamma)$ on $q$ in the second term, we will use that it can be well-approximated by $L(1,\chi_{q^*})$ for most sufficiently small $q$.

Indeed, set $\beta:=\f{64}{\gamma}$. Since we have $(N/q)^{\gamma}\geq(\log{N})^{64}$ for all $q\leq\f{N}{(\log{N})^\beta}$, it follows from Proposition~2.2 of~\cite{GranvilleSound2003} with $A=4$ and $D=N$ that $L(1,\chi_{q^*};(N/q)^\gamma)\ll L(1,\chi_{q^*})$ for all but at most $N^{\f{1}{2}}$ moduli $N^{\f{2}{3}}<q\leq \f{N}{(\log{N})^\beta}$. Using the trivial bound $L(1,\chi;(N/q)^\gamma)\ll\log{(N/q)}$, the contribution to the sum~\eqref{eq4.4} coming from these $N^{\f{1}{2}}$ moduli is $\ll N^{-\f{1}{6}}$, which is more than admissible. It thus remains to bound
  \[
\sum_{N^{\f{2}{3}}<q\leq \f{N}{(\log{N})^{\beta}}}\f{L(1,\chi_{q^*})^{\f{1}{2}}}{q(\log(N/q))^{\f{3}{2}}}
  \]
by the positivity of $L(1,\chi_{q^*})$. Setting $f(t):=\f{1}{t\log(N/t)^{3/2}}$, we apply partial summation to bound the sum above by
\[
f\left(\f{N}{(\log{N})^\beta}\right)\sum_{q\leq \f{N}{(\log{N})^{\beta}}}L(1,\chi_{q^*})^{\f{1}{2}}-\int_{N^{\f{2}{3}}}^{\f{N}{(\log{N})^\beta}}f'(t)\sum_{q\leq t}L(1,\chi_{q^*})^{\f{1}{2}}dt,
\]
which, by Lemma~\ref{lem3.3}, is
\[
\ll\f{1}{\log{N}(\log_2{N})^{\f{3}{2}}}+\int_{N^{\f{2}{3}}}^{\f{N}{(\log{N})^\beta}}\f{1}{t(\log{t})\log(N/t)^{\f{3}{2}}}dt.
\]
Noting that
\[
  \int_{N^{\f{2}{3}}}^{\f{N}{(\log{N})^\beta}}\f{1}{t(\log{t})\log(N/t)^{\f{3}{2}}}dt\ll\f{1}{\log{N}}\int_{N^{\f{2}{3}}}^{\f{N}{(\log{N})^\beta}}\f{1}{t\log(N/t)^{\f{3}{2}}}dt\ll\f{1}{\log{N}(\log_2{N})^{\f{1}{2}}}
\]
completes the proof of the lemma.
\end{proof}

Next, we bound the size of~\eqref{eq4.2}.
\begin{lemma}\label{lem4.3}
  We have
  \[
\#\left\{n\in S(N):n^2+n+1=q_1q_2\text{ with }q_1<q_2\text{ and }\f{N}{(\log{N})^{\beta}}<q_1\leq\f{N}{(\log{N})^{\f{1}{2}}}\right\}\ll_{\beta}\f{N}{\log{N}\exp(C\f{\log_5{N}}{\log_6{N}})},
  \]
  for some absolute constant $C>0$.
\end{lemma}
\begin{proof}
  We begin by splitting $n$ in~\eqref{eq4.2} up based on the smallest $k=k(N)$ prime factors of $n$. Note that the size of~\eqref{eq4.2} is at most
\begin{align*}
&\sum_{p_1<\dots<p_k\leq\log_3{N}}\# T_{p_1,\dots,p_k}(N) \\
  &\qquad+\#\{n\leq N:p_1(n^2+n+1)\geq N^{\f{2}{3}}\text{ and }\omega(n)<k\}\\
  &\qquad+\#\{n\leq N:p_1(n^2+n+1)\geq N^{\f{2}{3}},\ \omega(n)\geq k,\text{ and }p_{k}(n)>\log_3{N}\},
\end{align*}
where
\begin{align*}
  T_{p_1,\dots,p_k}(N):=\{n\in S(N):n^2+n+1=q_1q_2\text{ with }q_1<q_2\text{ and }\ \f{N}{(\log{N})^{\beta}}<q_1\leq\f{N}{(\log{N})^{\f{1}{2}}},&\\
  \omega(n)\geq k,\text{ and }p_i(n)=p_i\text{ for each }i=1,\dots,k&\}.
\end{align*}

The quantity $\#\{n\leq N:p_1(n^2+n+1)\geq N^{\f{2}{3}}\text{ and }\omega(n)<k\}$ is bounded above by
\[
\sum_{j=1}^{k-1}\sum_{\substack{m\leq N^{\f{j-1}{j}} \\ \omega(m)=j-1}}\#\left\{n\leq\f{N}{m}:p_1\left((nm)^2+nm+1\right)\geq N^{\f{2}{3}}\text{ and }\omega(n)=1\right\},
\]
which, by an application of an upper bound sieve, is $\ll\f{N}{(\log{N})^2}(C\log_2{N})^k$
for some absolute constant $C>0$. If $k\leq \f{\log_2{N}}{3\log_3{N}}$, then the right-hand side of the above inequality is certainly $\ll\f{N}{(\log{N})^{3/2}}$, which is admissible. A similar argument shows that the quantity
\[
\#\{n\leq N:p_1(n^2+n+1)\geq N^{\f{2}{3}},\ \omega(n)\geq k,\text{ and }p_{k}(n)>\log_3{N}\}
\]
is $\ll\f{N(C'\log_5{N})^{k}}{\log{N}\log_4{N}}$ for some absolute constant $C'>0$. If $k\leq \f{\log_5{N}}{3\log_6{N}}$, the right-hand side of the above inequality is $\ll\f{N}{\log{N}(\log_4{N})^{1/2}}$, which is, again, admissible. So assume, for the remainder of the proof, that $k=\lfloor \f{\log_5{N}}{3\log_6{N}}\rfloor$.

It remains to bound $\sum_{p_1<\dots<p_k\leq\log_3{N}}\# T_{p_1,\dots,p_k}(N)$. We apply Theorem~\ref{thm2.1} and split each $n\in T_{p_1,\dots,p_k}(N)$ up based on the prime factor $q$ of $n^2+n+1$ lying in the interval $[\f{N}{(\log{N})^\beta},\f{N}{(\log{N})^{1/2}}]$ to get
\[
\#T_{p_1,\dots,p_k}(N)\leq \sum_{\substack{\f{N}{(\log{N})^\beta}\leq q\leq\f{N}{(\log{N})^{1/2}} \\\left(\f{p_i}{q}\right)=1,\ i=1,\dots,k}}\# T^q_{p_1,\dots,p_k}(N)+O(N^{\f{1}{2}}),
\]
where
\[
T_{p_1,\dots,p_k}^q(N):=\{n\in T_{p_1,\dots,p_k}(N):q\mid n^2+n+1\}.
\]

Since $\prod_{p<\log_3{N}}p\ll(\log_2{N})^{1+o(1)}$, by an application of an upper bound sieve, we have
\[
\#T_{p_1,\dots,p_k}^{q}(N)\ll\f{N}{qp_1\cdots p_k}\prod_{\substack{p<p_k \\ p\neq p_i}}\left(1-\f{1}{p}\right)\f{\log{k}}{\log_2{N}}.
\]
Summing over $q\in[\f{N}{(\log{N})^\beta},\f{N}{(\log{N})^{1/2}}]$ such that $\left(\f{p_i}{q}\right)=1$ for each $i=1,\dots,k$ yields
\[
  \# T_{p_1,\dots,p_k}(N)\ll_\beta\f{\log{k}}{2^k}\f{1}{p_1\cdots p_k}\prod_{\substack{p<p_k \\ p\neq p_i}}\left(1-\f{1}{p}\right)\f{N}{\log{N}},
\]
and then summing over $p_1<\dots<p_k\leq\log_3{N}$ yields
\[
  \sum_{p_1<\dots<p_k\leq\log_3{N}}\# T_{p_1,\dots,p_k}(N)\ll_\beta\f{\log{k}}{2^k}\f{N}{\log{N}}\sum_{p_1<\dots<p_k\leq\log_3{N}}\f{\prod_{\substack{p<p_k \\ p\neq p_i}}\left(1-\f{1}{p}\right)}{p_1\cdots p_k}\ll_\beta\f{\log{k}}{2^k}\f{N}{\log{N}},
\]
where for the last inequality we use the fact that $\sum_{p_1<\dots<p_k\leq\log_3{N}}\f{\prod_{\substack{p<p_k \\ p\neq p_i}}\left(1-\f{1}{p}\right)}{p_1\cdots p_k}$ is the density of integers with $k$ prime factors below $\log_3{N}$, which, being a density, forces it to be at most $1$. Recalling our choice of $k$ now gives the conclusion of the lemma.
\end{proof}

Finally, we bound the size of~\eqref{eq4.3} using an enveloping sieve argument. We will require the following lattice point counting lemma, whose proof we defer to Section~\ref{sec6}.

\begin{lemma}\label{lem4.4}
There exist absolute constants $C>0$ and $0<\delta_1,\delta_1'<1$ such that the following holds.  For all $k,\ell\leq X^{\delta_1'}$ with $2,3\nmid k\ell$, we have
  \[
    \#\{(x,y,z)\in\Z^3:y^2-4xz=-3,\ 1\leq x,z\leq X,\ k\mid x,\text{ and }\ell\mid z\}= C\rho'(k\ell)\f{X}{k\ell}+O(X^{1-\delta_1}),
  \]
  where $\rho'$ is the multiplicative function defined by $\rho'(n):=\rho(n)\prod_{p\mid n}\left(1+\f{\chi_{-3}(p)}{p}\right)\1$ with $\rho(n):=\#\{a\Mod{n}:a^2\equiv -3\Mod{n}\}$.
\end{lemma}

Lemma~\ref{lem4.5} below gives a bound for~\eqref{eq4.3} after forgetting the condition $n\in S(N)$ and making a change of variables.
\begin{lemma}\label{lem4.5}
  Let $X>0$. Then
  \begin{equation}\label{eq4.5}
\#\left\{(p_1,p_2)\in\left[X^{\f{1}{2}},X\right]^2: 4p_1p_2-3\text{ is a perfect square}\right\}\ll\f{X}{(\log{X})^2}.
  \end{equation}
\end{lemma}
\begin{corollary}\label{cor4.6}
  We have
  \[
\#\left\{n\leq N: n^2+n+1=q_1q_2\text{ with }q_1<q_2\text{ and }\f{N}{(\log{N})^{\f{1}{2}}}<q_1\leq N\right\}\ll\f{N}{(\log{N})^{\f{3}{2}}}.
  \]
\end{corollary}
\begin{proof}
  This is an immediate consequence of Lemma~\ref{lem4.5} with $X=N(\log{N})^{1/2}$. Indeed, if $n^2+n+1=p_1p_2$, then, by completing the square and multiplying through by $4$, we have $(2n+1)^2+3=4p_1p_2$. Thus, each $n$ for which $n^2+n+1$ is the product of two primes corresponds to a pair of primes $(p_1,p_2)$ for which $4p_1p_2-3$ is a perfect square. If $n\leq N$ and $p_1\geq \f{N}{(\log{N})^{1/2}}$, then $p_1<p_2\leq N(\log{N})^{1/2}$, and so Lemma~\ref{lem4.5} applies.
\end{proof}

\begin{proof}[Proof of Lemma~\ref{lem4.5}]
This is a straightforward application of the enveloping sieve, and our argument will be closely modeled after those given in Section~2 of~\cite{Tao2006}.

Set $\gamma:=\min\left(\delta_1',\f{\delta_1}{8}\right)$ and $Y:=X^{\gamma}$, with $\delta_1$ and $\delta_1'$ as in Lemma~\ref{lem4.4}. Fix a smooth function $\phi:\R\to\R$ supported on $[-1,1]$ with $\phi(0)=1$ and $\int_0^1|\phi'(t)|^2dt=1$, and set
  \[
\nu(n):=\left(\sum_{k\mid n}\mu(k)\phi\left(\f{\log{k}}{\log{Y}}\right)\right)^2.
  \]
  Note that $\nu(p)=1$ whenever $p> Y$ is prime. So, since $\nu$ is nonnegative, the left-hand side of~\eqref{eq4.5} is bounded above by
  \begin{equation}\label{eq4.6}
\sum_{\substack{n,m\leq X\\ 2,3\nmid nm}}\nu(n)\nu(m)1_{\square}(4nm-3),
  \end{equation}
  where $1_{\square}$ denotes the indicator function of the squares. Expanding the definition of $\nu$ and applying Lemma~\ref{lem4.4}, we get that~\eqref{eq4.6} equals $\f{C^2}{4}$ times
\begin{align*}
  &X\sum_{\substack{k_1,k_2,\ell_1,\ell_2 \\ 2,3\nmid k_1k_2\ell_1\ell_2}}\f{\mu(k_1)\mu(k_2)\mu(\ell_1)\mu(\ell_2)}{[k_1,k_2][\ell_1,\ell_2]}\rho'([k_1,k_2][\ell_1,\ell_2])\phi\left(\f{\log{k_1}}{\log{Y}}\right)\phi\left(\f{\log{k_2}}{\log{Y}}\right)\phi\left(\f{\log{\ell_1}}{\log{Y}}\right)\phi\left(\f{\log{\ell_2}}{\log{Y}}\right)\\
  &+O\left(X^{1-\delta_1}\sum_{k_1,k_2,\ell_1,\ell_2}\left|\phi\left(\f{\log{k_1}}{\log{Y}}\right)\phi\left(\f{\log{k_2}}{\log{Y}}\right)\phi\left(\f{\log{\ell_1}}{\log{Y}}\right)\phi\left(\f{\log{\ell_2}}{\log{Y}}\right)\right|\right).
\end{align*}
Note that, since $\phi$ is supported on $[-1,1]$, the above sums over $k_1,k_2,\ell_1,$ and $\ell_2$ run over at most $Y^{4}\leq X^{\f{\delta_1}{2}}$ quadruples of integers. Thus, the error term is $\ll_{\phi}X^{1-\f{\delta_1}{2}}$.

We now focus on the main term. As in the proof of Proposition~2.1 in~\cite{Tao2006}, we apply Fourier inversion to write
\[
e^t\phi(t)=\int_{-\infty}^\infty\psi(u)e^{-itu}du
\]
for $\psi$ rapidly decaying. It then follows, by the rapid decay of $\psi$, that
\[
\phi\left(\f{\log{k}}{\log{Y}}\right)=\int_{|u|\leq(\log{Y})^{1/2}}\f{\psi(u)}{k^{\f{1+iu}{\log{Y}}}}du+O\left(\f{k^{-\f{1}{\log{Y}}}}{(\log{Y})^{10}}\right)
\]
for any $k\geq 1$, so that the main term of our expression for~\eqref{eq4.6} equals
\begin{align*}
\f{C^2}{4}X\iiiint\displaylimits_{|u_i|\leq(\log{Y})^{1/2}}\prod_{i=1}^4\psi(u_i)\left[\sum_{\substack{k_1,k_2,\ell_1,\ell_2 \\ 2,3\nmid k_1k_2\ell_1\ell_2}}\f{\mu(k_1)\mu(k_2)\mu(\ell_1)\mu(\ell_2)\rho'([k_1,k_2][\ell_1,\ell_2])}{[k_1,k_2][\ell_1,\ell_2]k_1^{\f{1+iu_1}{\log{Y}}}k_2^{\f{1+iu_2}{\log{Y}}}\ell_1^{\f{1+iu_3}{\log{Y}}}\ell_2^{\f{1+iu_4}{\log{Y}}}}\right]du_1du_2du_3du_4,
\end{align*}
plus an error term that is $O(\f{X}{(\log{N})^6})$.

By Hensel's lemma, $\rho'$ evaluated at any integer $k$ with $2,3\nmid k$ equals $\rho'$ evaluated at the squarefree part of $k$. Thus, the quantity inside of the brackets above can be expressed as the Euler product
\begin{align*}
  \prod_{p>3}\bigg(1-&\f{\rho'(p)}{p^{1+\f{(1+iu_1)}{\log{Y}}}}-\f{\rho'(p)}{p^{1+\f{(1+iu_2)}{\log{Y}}}}-\f{\rho'(p)}{p^{1+\f{(1+iu_3)}{\log{Y}}}}-\f{\rho'(p)}{p^{1+\f{(1+iu_4)}{\log{Y}}}}+\f{\rho'(p)}{p^{1+\f{2+iu_1+iu_2}{\log{Y}}}}+\f{\rho'(p)}{p^{1+\f{2+iu_3+iu_4}{\log{Y}}}}\\
  +&\f{\rho'(p)}{p^{2+\f{2+iu_1+iu_3}{\log{Y}}}}+\f{\rho'(p)}{p^{2+\f{2+iu_1+iu_4}{\log{Y}}}}+\f{\rho'(p)}{p^{2+\f{2+iu_2+iu_3}{\log{Y}}}}+\f{\rho'(p)}{p^{2+\f{2+iu_2+iu_4}{\log{Y}}}}-\f{\rho'(p)}{p^{2+\f{3+iu_1+iu_2+iu_3}{\log{Y}}}} \\
-&\f{\rho'(p)}{p^{2+\f{3+iu_1+iu_2+iu_4}{\log{Y}}}}-\f{\rho'(p)}{p^{2+\f{3+iu_1+iu_3+iu_4}{\log{Y}}}}-\f{\rho'(p)}{p^{2+\f{3+iu_2+iu_3+iu_4}{\log{Y}}}}+\f{\rho'(p)}{p^{2+\f{4+iu_1+iu_2+iu_3+iu_4}{\log{Y}}}}\bigg).
\end{align*}
Letting $L(s):=\sum_{n}\f{\mu(n)\rho'(n)}{n^s}$ denote the Dirichlet series for $\mu\cdot\rho'$, this Euler product equals
\begin{equation}\label{eq4.7}
\f{L(1+\f{(1+iu_1)}{\log{Y}})L(1+\f{(1+iu_2)}{\log{Y}})L(1+\f{(1+iu_3)}{\log{Y}})L(1+\f{(1+iu_4)}{\log{Y}})}{L(1+\f{2+iu_1+iu_2}{\log{Y}})L(1+\f{2+iu_3+iu_4}{\log{Y}})}\left(C'+O\left(\f{1+\sum_{i=1}^4|u_i|}{\log{Y}}\right)\right)
\end{equation}
for some absolute constant $C'>0$.

Note that, by the definition of $\rho'$, we have
\[
L(s)=\prod_{p>3}\left(1-\f{1+\chi_{-3}(p)}{p^s+p^{s-1}}\right)=\prod_{p>3}\left(1-\f{1+\chi_{-3}(p)}{p^s}-\f{1+\chi_{-3}(p)}{p^{2s+1}+p^s}\right),
\]
when $\Re(s)>1$, so that
\[
L(s)=\f{D(s)}{\zeta(s)L(s,\chi_{-3})}
\]
for some function $D(s)$ that is holomorphic in the region $\Re(s)>1/2$ and nonvanishing in a neighborhood of $s=1$. Thus, $L(s)=c_{\rho'}(s-1)+O(|s-1|^2)$ for some nonzero constant $c_{\rho'}$, since $\zeta(s)=\f{1}{s-1}+O(1)$ and $D(s)$ and $L(s,\chi_{-3})$ are holomorphic and nonvanishing in a neighborhood of $s=1$. As a consequence, when $|u_i|\leq(\log{Y})^{1/2}$ for each $i=1,2,3,4$, the quantity~\eqref{eq4.7} equals
\[
\f{c_{\rho'}^2}{(\log{Y})^2}\cdot\f{(1+iu_1)(1+iu_2)(1+iu_3)(1+iu_4)}{(2+iu_1+iu_2)(2+iu_3+iu_4)}\left(C'+O\left(\f{1}{(\log{Y})^{1/2}}\right)\right).
\]

The main term of our expression for~\eqref{eq4.6} thus equals
\begin{align*}
  &\f{(Cc_{\rho'})^2C'}{4}\cdot\f{X}{(\log{Y})^2}\iiiint\displaylimits_{|u_i|\leq(\log{Y})^{1/2}}\prod_{i=1}^4\psi(u_i)\f{(1+iu_1)(1+iu_2)(1+iu_3)(1+iu_4)}{(2+iu_1+iu_2)(2+iu_3+iu_4)}du_1du_2du_3du_4 \\
  &+O\left(\f{X}{(\log{Y})^{5/2}}\iiiint\displaylimits_{|u_i|\leq(\log{Y})^{1/2}}\prod_{i=1}^4|\psi(u_i)|\f{(1+|u_1|)(1+|u_2|)(1+|u_3|)(1+|u_4|)}{(2+|u_1|+|u_2|)(2+|u_3|+|u_4|)}du_1du_2du_3du_4\right).
\end{align*}
The error term is $\ll\f{X}{(\log Y)^{5/2-\ve}}$ since $\psi$ is rapidly decaying, and, by extending the integral in the main term to all of $\R^4$ using the rapid decay of $\psi$, the main term equals
\[
\f{(Cc_{\rho'})^2C'}{4}\cdot\f{X}{(\log{Y})^2}\left(\int_{-\infty}^{\infty}\int_{-\infty}^{\infty}\psi(u_1)\psi(u_2)\f{(1+iu_1)(1+iu_2)}{(2+iu_1+iu_2)}du_1du_2\right)^2,
\]
plus an error that is $O(\f{X}{(\log{Y})^{10}})$, say. The above quantity equals $\f{(Cc_{\rho'})^2C'}{4}\cdot\f{X}{(\log{Y})^2}$, since the double integral equals $1$ (see the manipulation at the end of the proof of Proposition~2.2 of~\cite{Tao2006}). The conclusion of the lemma now follows from our choice of $Y$.
\end{proof}

Proposition~\ref{prop4.1} is now an immediate consequence of Lemmas~\ref{lem4.2} and~\ref{lem4.3} and Corollary~\ref{cor4.6}.
\section{$n^2+n+1$ is prime}\label{sec5}
In this section, we bound the number of $n\in S(N)$ such that $n^2+n+1$ is prime:
\begin{proposition}\label{prop5.1}
  We have
  \[
\#\{n\in S(N):n^2+n+1\text{ is prime}\}\ll\f{N}{\log{N}\exp(C\f{\log_5{N}}{\log_6{N}})}
  \]
  for some absolute constant $C>0$.
\end{proposition}

The number of $n\in S(N)$ such that $n^2+n+1$ is a prime that is congruent to $3$ modulo $4$ is easy to bound.
\begin{lemma}\label{lem5.2}
  We have
  \[
\#\{n\in S(N):n^2+n+1\text{ is prime and congruent to }3\Mod{4}\}\ll\f{N}{(\log{N})^{\f{3}{2}}}.
  \]
\end{lemma}
\begin{proof}
  By quadratic reciprocity, we have that if $p\mid n$ is odd and $n^2+n+1\equiv 3\Mod{4}$ is prime, then $\left(\f{p}{n^2+n+1}\right)=(-1)^{\f{p-1}{2}}\left(\f{n^2+n+1}{p}\right)=(-1)^{\f{p-1}{2}}$. Thus, by Theorem~\ref{thm2.1}, we have that the number of $n\in S(N)$ such that $n^2+n+1$ is a prime that is congruent to $3$ modulo $4$ is at most
  \[
\#\{n\leq N:\Omega(n^2+n+1)=1\text{ and }2\neq p\mid n\implies p\equiv 1\Mod{4}\}+O(N^{\f{1}{2}}).
  \]
The first term above is $\ll\f{N}{(\log{N})^{3/2}}$ by an upper bound sieve.
\end{proof}

It now remains to deal with $n\in S(N)$ such that $n^2+n+1$ is a prime that is congruent to $1$ modulo $4$. As outlined in Section~\ref{sec2}, to finish our proof of Theorem~\ref{thm1.2}, we will combine Lemma~\ref{lem2.2} with the quartic reciprocity law to reduce the problem of bounding the number of such $n$ to that of bounding the number of certain representations of prime values of $n^2+n+1$ by the quadratic form $x^2+y^2$. This can be done using the Selberg sieve as long as we have a sufficiently accurate count for the number of lattice points on the hyperboloid $4x^2+16y^2-z^2=3$ with $y$ and $z$ satisfying a variety of congruence restrictions and $|z|\leq 2N+1$.

We first state the quartic reciprocity law and the required lattice point counting lemma, whose proof we defer to Section~\ref{sec7}.

\begin{theorem}[Quartic reciprocity]\label{thm5.3}
  Let $q\equiv 1\Mod{4}$ and $p$ be primes satisfying $\left(\f{q}{p}\right)=1$  and let $\sigma$ be a root of the congruence $q\equiv\sigma^2\Mod{p}$. Assume that $q=x^2+y^2$ with $2\mid y$. Then
  \[
\left(\f{p^*}{q}\right)_4=\left(\f{\sigma(\sigma+y)}{p}\right).
  \]
\end{theorem}
(See Theorem~5.5 of~\cite{Lemmermeyer2000}.) As it will be relevant in the proof of Lemma~\ref{lem5.5} below, note that in the situation of Theorem~\ref{thm5.3}, we must have $\left(\f{\sigma(\sigma+y)}{p}\right)=\left(\f{\sigma(\sigma-y)}{p}\right)$.
\begin{lemma}\label{lem5.4}
  There exist absolute constants $C>0$ and $0<\delta_2,\delta_2'<1$ such that the following holds. Let $p_1,\dots,p_k\leq\log_3{N}$ and $q_1,\dots,q_{m}\leq\log_3{N}$ be disjoint collections of primes. For all odd squarefree $\ell\leq X^{\delta_2'}$ with $p_1\cdots p_kq_1\cdots q_m\mid \ell$ and for all congruence classes $a\Mod{p_1\cdots p_k}$ and $b\Mod{8\ell}$ satisfying
\begin{enumerate}
\item $\left(\f{1+2a}{p_i}\right)=\left(\f{1-2a}{p_i}\right)$ for each $i=1,\dots,k$,
\item $b\equiv 1\Mod{2p_i}$ for each $i=1,\dots,k$,
\item $b\equiv 1\Mod{2q_j}$ for each $j=1,\dots,m$, and
\item $b^2\equiv -3\Mod{r}$, where $r:=\f{\ell}{p_1\cdots p_kq_1\cdots q_m}$,
\end{enumerate}
  we have
  \begin{align*}
    \#\{(x,y,z)\in\Z^3:&4x^2+16y^2-z^2=3,\ |z|\leq X,\ y\equiv a\Mod{p_1\cdots p_k},\text{ and }z\equiv b\Mod{8\ell}\} \\
    &=C2^k\prod_{i=1}^kw_1(a,p_i)\prod_{j=1}^mw_2(q_j)\prod_{p\mid r}w_3(p)\f{X}{(p_1\cdots p_k)^2q_1\cdots q_mr}+O(X^{1-\delta_2}),
  \end{align*}
  where
  \[
    w_1(a,p)=(1+\chi_{12}(p))^{\delta_{\f{1}{4}}(a\Mod{p})}\left(\f{1-\f{\chi_{12}(p)}{p}}{1-\f{1}{p^2}}\right),
  \]
  \[
w_2(p):=\left(1+\f{2\chi_{12}(p)-1}{p}\right)\left(\f{1-\f{\chi_{12}(p)}{p}}{1-\f{1}{p^2}}\right),
  \]
  and
  \[
    w_3(p):=(1+\chi_{-1}(p))\left(1+\f{2\chi_{12}(p)-1}{2p}\right)\left(\f{1-\f{\chi_{12}(p)}{p}}{1-\f{1}{p^2}}\right).
  \]
\end{lemma}
Now we can bound the number of $n\in S(N)$ such that $n^2+n+1$ is a prime that is congruent to $1$ modulo $4$.
\begin{lemma}\label{lem5.5}
  We have
  \[
   \#\{n\in S(N):n^2+n+1\text{ is prime and congruent to }1\Mod{4}\}\ll\f{N}{\log{N}\exp(C\f{\log_5{N}}{\log_6{N}})}
  \]
  for some absolute constant $C>0$.
\end{lemma}
\begin{proof}
The proof begins in the same manner as the proof of Lemma~\ref{lem4.3}. We split $n$ up based on the smallest $k=k(N)$ prime factors of $n$ to get that the size of the set in question is at most
  \begin{align*}
    &\sum_{p_1<\dots<p_k\leq\log_3{N}}\# T_{p_1,\dots,p_k}(N) \\
    &+\#\{n\leq N: n^2+n+1\text{ is prime and }\omega(n)<k\} \\
    &+\#\{n\leq N: n^2+n+1\text{ is prime},\ \omega(n)\geq k,\text{ and }p_k(n)>\log_3{N}\},
  \end{align*}
  where
  \[
T_{p_1,\dots,p_k}(N):=\{n\in S(N):n^2+n+1\equiv 1\Mod{4}\text{ is prime},\ \omega(n)\geq k,\text{ and } p_i(n)=p_i\text{ for }i=1,\dots,k\}.
  \]
  By the argument in the proof of Lemma~\ref{lem3.2}, by setting $k=\lfloor\f{\log_5{N}}{3\log_6{N}}\rfloor$ we get that the sizes of the two sets above are $\ll\f{N}{(\log{N})^{3/2}}$ and $\ll\f{N}{\log{N}(\log_4{N})^{1/2}}$, respectively. So it remains to bound the sum of the $\#T_{p_1,\dots,p_k}(N)$'s.

By Lemma~\ref{lem2.2}, we have that $T_{p_1,\dots,p_k}(N)\subset T_{p_1,\dots,p_k}'(N)$, where
  \begin{align*}
    T'_{p_1,\dots,p_k}(N):=\bigg\{n\leq N:&\ n^2+n+1\equiv 1\Mod{4}\text{ is prime}\\
    &\text{and }p_i(n)=p_i\text{ and }\left(\f{p_i}{n^2+n+1}\right)_4=1\text{ for }i=1,\dots,k\bigg\}.
  \end{align*}
  Applying Theorem~\ref{thm5.3} with $q=n^2+n+1$ for $n\in T_{p_1,\dots,p_k}'(N)$ and $\sigma=1$ gives
  \[
\left(\f{p_i}{n^2+n+1}\right)_4=(-1)^{\f{p_i-1}{2}\cdot\f{n^2+n}{4}}\left(\f{p_i^*}{n^2+n+1}\right)_4=(-1)^{\f{p_i-1}{2}\cdot\f{n^2+n}{4}}\left(\f{1+2y}{p_i}\right)
  \]
  when $q=x^2+4y^2$. So, splitting $n\in T_{p_1,\dots,p_k}'(N)$ up based on whether $n\equiv 0,7\Mod{8}$ or $n\equiv 3,4\Mod{8}$, we see that $\left(\f{p_i}{n^2+n+1}\right)_4=1$ if and only if $n^2+n+1$ can be written in the form $x^2+4y^2$ with $y\equiv a\Mod{p_i}$ for some $a\Mod{p_i}$ such that both $1+2a$ and $1-2a$ are quadratic residues (if $(-1)^{\f{p_i-1}{2}\cdot\f{n^2+n}{4}}=1$) or nonresidues (if $(-1)^{\f{p_i-1}{2}\cdot\f{n^2+n}{4}}=-1$) modulo $p_i$. We will show that
\begin{equation}\label{eq5.1}
\#T_{p_1,\dots,p_k;\mathbf{a},j}'(N)  \ll 2^k\log{k}\prod_{i=1}^kw_1(a_i,p_i)\prod_{\substack{p< p_k \\ p\neq p_i}}\left(1-\f{1}{p}\right)\f{N}{(p_1\cdots p_k)^2\log{N}},
\end{equation}
where
\begin{align*}
T_{p_1,\dots,p_k;\mathbf{a},j}'(N):=\{n\leq N:&\ n^2+n+1\text{ prime},n\equiv j\Mod{8},\ p_i(n)=p_i\text{ for }i=1,\dots,k,\\
  &\text{ and } n^2+n+1=x^2+4y^2\text{ with }y\equiv a_i\Mod{p_i}\text{ for }i=1,\dots,k\}
\end{align*}
for such $\mathbf{a}=(a_1,\dots,a_k)$ and $j=0,3,4,7$.

For each $n\in\N$, set
\[
m_{\mathbf{a}}(n):=\#\{(x,y)\in\Z^2:n^2+n+1=x^2+4y^2\text{ with }y\equiv a_i\Mod{p_i}\text{ for }i=1,\dots,k\},
\]
so that
\[
\#T_{p_1,\dots,p_k;\mathbf{a},j}'(N)\ll\sum_{\substack{n\leq N \\ n^2+n+1\text{ prime } \\ n\equiv j\Mod{8}\\ p_i(n)=p_i\text{ for }i=1,\dots,k}}m_{\mathbf{a}}(n).
\]
For each $r\in\N$ and congruence class $b\Mod{r}$, also set
\[
M_{\mathbf{a},j}(r):=\sum_{\substack{n\leq N\\n\equiv b\Mod{r} \\ n\equiv j\Mod{8} \\ p_i(n)=p_i\text{ for }i=1,\dots,k}}m_{\mathbf{a}}(n).
\]
Note that, if $n^2+n+1=x^2+4y^2$, then, by completing the square, we have that $4x^2+16y^2-(2n+1)^2=3$ . So, by applying Lemma~\ref{lem5.4} when $r\leq N^{\f{\delta'_2}{2}}$ is odd, squarefree, and satisfies $(p_1\cdots p_k,r)=1$ and $b\Mod{r}$ satisfies $b^2\equiv -3\Mod{r}$, we have that
\[
M_{\mathbf{a},j}(r)=C2^k\prod_{j=1}^kw_1(a_j,p_j)\prod_{\substack{p<p_k \\ p\neq p_j}}\left(1-\f{w_2(p)}{p}\right)\prod_{p\mid r}w_3(p)\f{N}{(p_1\cdots p_k)^2r}+O\left(N^{1-\delta_2}\right)
\]
for each choice of $\mathbf{a}=(a_1,\dots, a_k)$ where $1+2a_i$ and $1-2a_i$ are both quadratic residues or nonresidues (depending on $j$ and $p_i$) modulo $p_i$ for $i=1,\dots,k$. Because of the power-saving error term above, we can apply the Selberg sieve to deduce (using that $w_2(p)=1+O(p\1)$ and $w_3(p)=1+\chi_{-1}(p)+O(p\1)$) the desired bound~\eqref{eq5.1} for each $\#T_{p_1,\dots,p_k;\mathbf{a},j}'(N)$.

Now, we sum over all admissible choices of $\mathbf{a}$. There are at most $\f{p_i+O(1)}{4}$ possible choices of $a_i\Mod{p_i}$ for each $p_i>2$ (by considering either the number of points on the conic $x^2+y^2=2$ modulo $p_i$ or the number of points on the conic $x^2+y^2=-2$ modulo $p_i$, since, for example, any $a$ for which $1+2a=y^2$ and $1-2a=x^2$ in $\Z/p\Z$ gives rise to a solution to $x^2+y^2=2$ modulo $p$). We thus have
\[
\sum_{\substack{a\Mod{p_i} \\ \left(\f{1-2a}{p_i}\right)=\left(\f{1+2a}{p_i}\right)=(-1)^{\f{p_i-1}{2}\cdot\f{j^2+j}{4}}}}\f{w_1(a,p_i)}{p_i^2}=\f{1}{4p_i}+O\left(\f{1}{p_i^2}\right),
\]
so, by the Chinese remainder theorem, we get that
\[
\# T'_{p_1,\dots,p_k}(N)\ll\f{\log{k}}{2^k}\f{1}{p_1\cdots p_k}\prod_{\substack{p<p_k \\ p\neq p_i}}\left(1-\f{1}{p}\right)\f{N}{\log{N}},
\]
and are in exactly the same situation as in the end of the proof of Lemma~\ref{lem4.3}. Summing over $p_1<\dots<p_k\leq\log_3{N}$ as in that argument yields the desired bound for the number of $n\in S(N)$ such that $n^2+n+1$ is a prime that is congruent to $1$ modulo $4$.
\end{proof}

Proposition~\ref{prop5.1} now follows from Lemmas~\ref{lem5.2} and~\ref{lem5.5}, and Theorem~\ref{thm1.2} follows from Propositions~\ref{prop3.1},~\ref{prop4.1}, and~\ref{prop5.1}.

\section{The first lattice point count}\label{sec6}
In this section, we prove Lemma~\ref{lem4.4}, the first of our two lattice point counting results. We do this by adapting an argument of Hooley~\cite{Hooley1963}, incorporating a bound of Duke, Friedlander, and Iwaniec~\cite{DukeFriedlanderIwaniec1995} in place of Hooley's bound for weighted averages of sums of additive characters over roots of quadratic congruences. For the convenience of the reader, we record Duke, Friedlander, and Iwaniec's result specialized to the case we will use.
\begin{proposition}[Duke, Friedlander, and Iwaniec, Proposition~4 of~\cite{DukeFriedlanderIwaniec1995}]\label{prop6.1}
  Let $f:\R\to\R$ be a function supported on $[X,2X]$ satisfying
  \[
|f^{(i)}(t)|\ll t^{-i}
  \]
  for each $i=0,\dots, 4$. Then, for all $h\ll X$, we have
  \[
\sum_{d\mid n}f(n)\sum_{\substack{ \nu^2\equiv -3\Mod{n} \\ 0<\nu\leq n}}e\left(h\f{\nu}{n}\right)\ll\tau(d)\left[1+\f{\tau(dh)(d,h)^{\f{1}{2}}}{d}X^{\f{1}{2}}\right]^{\f{1}{2}}X^{\f{1}{2}}\log{2X}.
  \]
\end{proposition}

Now we can prove Lemma~\ref{lem4.4}.
\begin{proof}[Proof of Lemma~\ref{lem4.4}]
We begin by writing the size of the set in Lemma~\ref{lem4.4} as
  \[
\sum_{\substack{x,z\leq X \\ k\mid x,\ \ell\mid z}}\sum_y 1_{y^2-4xz=-3}=\sum_{\substack{x\leq X \\ k\mid x}}\sum_{y^2\equiv -3\Mod{4\ell x}} 1_{\left[\f{X}{\ell}\right]}\left(\f{y^2+3}{4\ell x}\right).
  \]
Splitting $y$ up based on its congruence class modulo $4\ell x$, the above equals
  \[
\sum_{\substack{x\leq X \\ k\mid x}}\sum_{\substack{\nu^2\equiv -3\Mod{4\ell x}\\ 0<\nu\leq 4\ell x}}\sum_{y\equiv\nu\Mod{4\ell x}}1_{\left[\f{X}{\ell}\right]}\left(\f{y^2+3}{4\ell x}\right).
\]

Note that $\f{y^2+3}{4\ell x}\leq\f{X}{\ell}$ if and only if $|y|\leq\sqrt{4xX-3}$. Letting $\psi(u):=\lfloor u\rfloor -u+\f{1}{2}$ denote the sawtooth function, it thus follows that
  \begin{align*}
    \sum_{y\equiv \nu\Mod{4\ell x}}1_{\left[\f{X}{\ell}\right]}\left(\f{y^2+3}{4\ell x}\right) &= 2\cdot\#\{y\in[\sqrt{4xX-3}]: y\equiv\nu\Mod{4\ell x}\} \\
    &=2\left(\left\lfloor \f{\sqrt{4xX-3}-\nu}{4\ell x}\right\rfloor-\left\lfloor\f{-\nu}{4\ell x}\right\rfloor\right) \\
    &= 2\left(\f{\sqrt{4xX-3}}{4\ell x}+\psi\left(\f{\sqrt{4xX-3}-\nu}{4\ell x}\right)-\psi\left(\f{-\nu}{4\ell x}\right)\right).
  \end{align*}
Thus, since $\sqrt{4xX-3}=2\sqrt{xX}+O\left(\f{1}{\sqrt{xX}}\right)$, our desired count equals
  \begin{equation}\label{eq6.1}
2\sum_{\substack{x\leq X \\ k\mid x}}\sum_{\substack{ \nu^2\equiv -3\Mod{4\ell x} \\ 0<\nu\leq 4\ell x}}\left[\f{X^{\f{1}{2}}}{2\ell x^{\f{1}{2}}}+\psi\left(\f{\sqrt{4xX-3}-\nu}{4\ell x}\right)-\psi\left(\f{-\nu}{4\ell x}\right)\right]+O(\log{X}).
\end{equation}

We first deal with the main term of~\eqref{eq6.1}:
\[
\f{X^{\f{1}{2}}}{\ell}\sum_{\substack{x\leq X \\ k\mid x}}\f{\rho(4\ell x)}{x^{\f{1}{2}}}.
\]
Note that we can write $\sum_{\substack{x\leq X \\ k\mid x}}\f{\rho(4\ell x)}{x^{1/2}}=2\ell^{\f{1}{2}}\sum_{\substack{x\leq 4\ell X \\ 4k\ell\mid x}}\f{\rho(x)}{x^{1/2}}$. So, set $L(s):=\sum_{a\equiv 0\Mod{4k\ell}}\f{\rho(a)}{a^s}$ for $\Re(s)>1$. It is shown in Subsection~12.1 of~\cite{DukeFriedlanderIwaniec2012} (the restriction there that $D>0$ is unnecessary for the relevant computation) that
\[
L(s)=\f{\zeta(s)L(s,\chi_{-3})}{\zeta(2s)}\f{\rho(4k\ell)}{(4k\ell)^s}\prod_{p\mid k\ell}\left(1+\f{\chi_{-3}(p)}{p^{s}}\right)\1,
\]
so that $L(s)$ is holomorphic for $\Re(s)\geq 1/2$ aside from a simple pole at $s=1$, where it has residue
\[
\f{6}{\pi^2}L(1,\chi_{-3})\f{\rho(4k\ell)}{4k\ell}\prod_{p\mid k\ell}\left(1+\f{\chi_{-3}(p)}{p}\right)\1.
\]

By partial summation, we have
\[
\sum_{\substack{x\leq 4\ell X \\ 4k\ell\mid x}}\f{\rho(x)}{x^{\f{1}{2}}}=\f{1}{(4\ell X)^{\f{1}{2}}}\sum_{\substack{x\leq 4\ell X \\ 4k\ell\mid x}}\rho(x)+\f{1}{2}\int_{1}^{4\ell X}t^{-\f{3}{2}}\sum_{\substack{x\leq t \\ 4k\ell\mid x}}\rho(x)dt,
\]
while, by a standard contour integration, we also have that
\[
\sum_{\substack{x\leq t \\ 4k\ell\mid x}}\rho(x)=\f{6}{\pi^2}L(1,\chi_{-3})\f{\rho(4k\ell)}{4k\ell}\prod_{p\mid k\ell}\left(1+\f{\chi_{-3}(p)}{p}\right)\1 t+O(t^{\f{3}{4}+\ve}).
\]
So,
\[
\sum_{\substack{x\leq X \\ 4k\ell\mid x}}\f{\rho(x)}{x^{\f{1}{2}}}=C\f{\rho'(k\ell)}{2k\ell^{\f{1}{2}}} X^{\f{1}{2}}+O(X^{\f{1}{4}+\ve})
\]
for some absolute constant $C>0$. Thus, the main term in~\eqref{eq6.1} equals $C\f{\rho'(k\ell)}{k\ell}X+O(X^{\f{3}{4}+\ve})$.

Now we can deal with the error term of~\eqref{eq6.1}:
\[
2\sum_{\substack{x\leq X \\ k\mid x}}\sum_{\substack{ \nu^2\equiv -3\Mod{4\ell x} \\ 0<\nu\leq 4\ell x}}\psi\left(\f{\sqrt{4xX-3}-\nu}{4\ell x}\right)-2\sum_{\substack{x\leq X \\ k\mid x}}\sum_{\substack{ \nu^2\equiv -3\Mod{4\ell x} \\ 0<\nu\leq 4\ell x}}\psi\left(\f{-\nu}{4\ell x}\right).
\]
Arguing as in Section~5 of~\cite{Hooley1963}, we use the Fourier expansion of $\psi$ to write
\[
\psi(u)=\f{1}{\pi}\sum_{h=1}^M\f{\sin(2\pi hu)}{h}+O\left(\min\left(1,\f{1}{M\|u\|}\right)\right)
\]
for some $1\leq M\leq X^{\f{1}{2}}$ to be chosen later, so that the two sums appearing in the error term equal
\[
\f{1}{\pi}\sum_{h=1}^M\sum_{\substack{x\leq X \\ k\mid x}}\sum_{\substack{ \nu^2\equiv -3\Mod{4\ell x} \\ 0<\nu\leq 4\ell x}}\f{\sin(2\pi h\f{\sqrt{4xX-3}-\nu}{4\ell x})}{h}+O\left(\sum_{\substack{x\leq X \\ k\mid x}}\sum_{\substack{ \nu^2\equiv -3\Mod{4\ell x} \\ 0<\nu\leq 4\ell x}}\min\left(1,\f{1}{M\|\f{\sqrt{4xX-3}-\nu}{4\ell x}\|}\right)\right)
\]
and
\[
\f{1}{\pi}\sum_{h=1}^M\sum_{\substack{x\leq X \\ k\mid x}}\sum_{\substack{ \nu^2\equiv -3\Mod{4\ell x} \\ 0<\nu\leq 4\ell x}}\f{\sin(2\pi h\f{-\nu}{4\ell x})}{h}+O\left(\sum_{\substack{x\leq X \\ k\mid x}}\sum_{\substack{ \nu^2\equiv -3\Mod{4\ell x} \\ 0<\nu\leq 4\ell x}}\min\left(1,\f{1}{M\|\f{-\nu}{4\ell x}\|}\right)\right).
\]
To estimate the main term of the first sum, we use the sine addition law to write it as
\[
\f{1}{\pi}\sum_{h=1}^M\sum_{\substack{x\leq X \\ k\mid x}}\sum_{\substack{ \nu^2\equiv -3\Mod{4\ell x} \\ 0<\nu\leq 4\ell x}}\f{\sin(2\pi h\f{\sqrt{4xX-3}}{4\ell x})\cos(2\pi h\f{\nu}{4\ell x})-\cos(2\pi h\f{\sqrt{4xX-3}}{4\ell x})\sin(2\pi h\f{\nu}{4\ell x})}{h}.
\]
Using that $\sum_{\substack{ \nu^2\equiv -3\Mod{4\ell x} \\ 0<\nu\leq 4\ell x}}\cos\left(2\pi h\f{\nu}{4\ell x}\right)=\sum_{\substack{ \nu^2\equiv -3\Mod{4\ell x} \\ 0<\nu\leq 4\ell x}}e\left(h\f{\nu}{4\ell x}\right)$ and $\sin(-t)=-\sin(t)$, the expression above equals
\begin{equation}\label{eq6.2}
\f{1}{\pi}\sum_{h=1}^{M}\f{1}{h}\sum_{\substack{x\leq X \\ k\mid x}}\sin\left(2\pi h\f{\sqrt{4xX-3}}{4\ell x}\right)\sum_{\substack{ \nu^2\equiv -3\Mod{4\ell x} \\ 0<\nu\leq 4\ell x}}e\left(h\f{\nu}{4\ell x}\right).
\end{equation}
To estimate the error term of the first sum, we use the Fourier expansion of the function $u\mapsto\min\left(1,\f{1}{M\|u\|}\right)$ given in Section~5 of~\cite{Hooley1963} combined with the cosine addition law to write the sum inside of the big-$O$ as
\begin{equation}\label{eq6.3}
\f{1}{2}C_0(M)\sum_{\substack{x\leq X \\ k\mid x}}\rho(4\ell x)+\sum_{h=1}^\infty C_h(M)\sum_{\substack{x\leq X \\ k\mid X}}\cos\left(2\pi h\f{\sqrt{4xX-3}}{4\ell x}\right)\sum_{\substack{ \nu^2\equiv -3\Mod{4\ell x} \\ 0<\nu\leq 4\ell x}}e\left(h\f{\nu}{4\ell x}\right),
\end{equation}
where $C_h(M)\ll \f{\log{M}}{M}$ for all $h\geq 0$ and $C_h(M)\ll \f{M}{h^2}$ for all $h\geq 1$.

The main term of the expression for the second sum in the error term of~\eqref{eq6.1} vanishes, and the quantity inside of the error term can, similarly to above, be written as
\begin{equation}\label{eq6.4}
\f{1}{2}C_0(M)\sum_{\substack{x\leq X \\ k\mid x}}\rho(4\ell x)+\sum_{h=1}^\infty C_h(M)\sum_{\substack{x\leq X \\ k\mid X}}\sum_{\substack{ \nu^2\equiv -3\Mod{4\ell x} \\ 0<\nu\leq 4\ell x}}e\left(h\f{\nu}{4\ell x}\right).
\end{equation}

Now we bound~\eqref{eq6.2},~\eqref{eq6.3}, and~\eqref{eq6.4}, starting with the portions of the sums that can be bounded trivially. Take $M=X^{1/1000}$. The contribution to~\eqref{eq6.2} coming from $x\leq X^{99/100}$ is $\ll_\ve\ell^\ve X^{99/100+\ve}$. The contribution to~\eqref{eq6.3} and~\eqref{eq6.4} coming from $h\geq \delta X$ is $\ll_\ve MX^\ve/\delta\ll_\ve X^{1/1000+\ve}/\delta$, and the contribution to the second sum in~\eqref{eq6.3} and~\eqref{eq6.4} coming from $x\leq X^{99/100}$ is $\ll_\ve MX^{99/100+\ve}\ll_\ve X^{991/1000+\ve}$.

To bound the remainder of~\eqref{eq6.2},~\eqref{eq6.3}, and~\eqref{eq6.4}, we will apply Proposition~\ref{prop6.1} on dyadic intervals. Fix a smooth function $\phi:\R\to[0,1]$ supported on $[4,8]$ with $\phi(t)=1$ for $x\in[5,7]$ that decreases to $0$ on $[4,5]$ and $[7,8]$. For any $Y>0$, define
\[
\phi_Y(t):=X^{-12/200}\phi(4t/Y),
\]
so that $\phi_Y:\R\to[0,X^{-12/200}]$ is supported on $[Y,2Y]$, is identically $X^{-12/200}$ on $[5Y/4,7Y/4]$, and, by the chain rule, satisfies
\[
\phi_Y^{(i)}(t)\leq X^{-12/200}\left(\f{4}{Y}\right)^{i}\max_{t'\in[4,8]}\phi^{(i)}(t')\ll X^{-12/200}Y^{-i}
\]
for all $t\in [Y,2Y]$ and $i=0,\dots,4$. Now, we apply Proposition~\ref{prop6.1} on each of the intervals $[Y_i,2Y_i]$ with $Y_i=4/7(5/7)^iX$ for each $i$ for which $Y_i\geq X^{99/100}/2$ and $f(t)$ equal to $\sin(2\pi h(\f{tX}{\ell}-3)^{1/2}/t)\phi_{Y_i}(t)$, $\cos(2\pi h(\f{tX}{\ell}-3)^{1/2}/t)\phi_{Y_i}(t)$, and $\phi_{Y_i}(t)$. The third choice of $f$ obviously satisfies the derivatives condition in Proposition~\ref{prop6.1}, and the other two do as well as long as $h\leq X^{1/100}$. This gives a bound of $\ll_\ve \ell^\ve X^{3/4+12/200+\ve}$ for the contribution of $x\geq X^{99/100}$ to~\eqref{eq6.2} and a bound of $\ll_\ve (k\ell)^\ve(\f{X\log{M}}{M}+X^{3/4+12/200+\ve})\ll_\ve (k\ell)^\ve X^{999/1000+\ve}$ for the contribution of $h\leq X^{1/100}$ and $x\geq X^{99/100}/2$ to~\eqref{eq6.3} and~\eqref{eq6.4}. Combining these bounds with the trivial contributions from the previous paragraph (choosing $\delta=X^{-99/100}$), we conclude that the error term of~\eqref{eq6.1} is $\ll (k\ell)^{\ve}X^{999/1000+\ve}$, completing the proof of the lemma.
\end{proof}

\section{The second lattice point count}\label{sec7}
We will prove Lemma~\ref{lem5.4} following the same strategy as the proof of Lemma~\ref{lem4.4}, with two key differences. The first stems from the fact that these two lemmas concern different hyperboloids, and so a change of variables is needed before the hyperboloid in Lemma~\ref{lem5.4} can be analyzed in a similar manner to the hyperboloid in Lemma~\ref{lem4.4}, which introduces additional complications. The second is that there is not, currently in the literature, any analogue of Proposition~\ref{prop6.1} that can be applied to the situation of Lemma~\ref{lem5.4}. We will prove such a result from scratch in Lemma~\ref{lem7.2}, adapting an argument of Hooley from Section~6 of~\cite{Hooley1963}. One of the ingredients of this proof is the following classical lemma, which connects roots of quadratic congruences to representations by quadratic forms.
\begin{lemma}\label{lem7.1}
  \begin{enumerate}
  \item Let $n\in\N$. There is a bijective correspondence between solutions $\nu\in\Z/n\Z$ to the congruence $\nu^2\equiv 3\Mod{n}$ and equivalence classes of primitive representations of $n$ by the quadratic form $x^2-3y^2$
\[
\{\nu\Mod{n}:\nu^2\equiv 3\Mod{n}\}\leftrightarrow\left\{(r,s)\in\Z^2:n=r^2-3s^2\right\}/\sim,
\]
where $(r,s)\sim (r',s')$ if $r'=ar+bs$ and $s'=cr+ds$ for some $(\begin{smallmatrix}a & b \\ c & d \end{smallmatrix})\in \Aut_{x^2-3y^2}(\Z)$, given by
  \[
\nu=r\rho-3s\sigma \leftrightarrow n=r^2-3s^2,
  \]
  where $r\sigma-s\rho=1$.   In this situation, we have
  \[
\f{\nu}{n}\equiv-\f{\overline{s}_r}{r}-\f{3s}{r(r^2-3s^2)}\Mod{1},
  \]
  where $\overline{s}_r$ denotes the multiplicative inverse of $s$ modulo $r$. In each equivalence class of $\sim$, there is exactly one representation $n=r^2-3s^2$ with $r,s>0$ and $s\leq\f{r}{2}$.
\item Let $n,m\in\N$ with $\gcd(n,m)=1$. Suppose that $n=a^2-3b^2$ and $m=r^2-3s^2$ are primitive representations of $n$ and $m$, respectively, and that $(ar+3bs)\sigma-(as+br)\rho=1$. Then $\nu:=(ar+3bs)\rho-3(as+br)\sigma$ satisfies
  \[
\f{\nu}{n}\equiv -\f{\overline{b}_a}{a}-\f{3b}{a(a^2-3b^2)}\Mod{1}.
  \]
  \end{enumerate}
\end{lemma}
\begin{proof}
  The proof of the first statement can be found in~\cite[Art. 86 and Art. 100]{Smith1894}, but we include an argument here as well. Since every binary quadratic form of discriminant $12$ is equivalent to $x^2-3y^2$, for every $0<\nu\leq n$ satisfying $\nu^2\equiv 3\Mod{n}$, there exists $(\begin{smallmatrix} r & \rho \\ s & \sigma \end{smallmatrix})\in \SL_2(\Z)$ such that $g(rx+\rho y,sx+\sigma y)=x^2-3y^2$, where $g(x,y)$ is the form
  \[
g(x,y):=nx^2+2\nu sy+\f{\nu^2-3}{n}y^2.
  \]
(Further, the set of such $(\begin{smallmatrix} r & \rho \\ s & \sigma \end{smallmatrix})$ is a coset of $\Aut_{x^2-3y^2}(\Z)$ in $\SL_2(\Z)$.) In this situation, we must have $n=r^2-3s^2$ and $\nu=r\rho-s\sigma$. That there is exactly one such matrix $(\begin{smallmatrix} r & \rho \\ s & \sigma \end{smallmatrix})$ satisfying $r,s>0$ and $s\leq\f{r}{2}$ follows from the fact that $2+\sqrt{3}$ is a fundamental unit of $\mathcal{O}_{\mathbb{Q}(\sqrt{3})}$, and the expression for $\f{\nu}{n}$ follows from a straightforward manipulation.

The second statement follows immediately from the fact that, whenever $\ell=t^2-3u^2$ for $\gcd(t,u)=1$ and $t\sigma'-u\rho'=t\sigma''-u\rho''=1$, we have that $t\rho'-3u\sigma'\equiv t\rho''-3u\sigma''\Mod{\ell}$. Indeed, note that the condition $(ar+3bs)\sigma-(as+br)\rho=1$ implies that $a(r\sigma-s\rho)-b(r\rho-3s\sigma)=1$ and the definition of $\nu$ can be rewritten as $\nu=a(r\rho-3s\sigma)-3b(r\sigma-s\rho)$.
\end{proof}

We now argue along the lines of Section~6 of~\cite{Hooley1963} to prove the following lemma.

\begin{lemma}\label{lem7.2}
There exists an absolute constant $0<\gamma<1$ such that the following holds. For every $|h|,|h'|\leq X^{\gamma}$, $k,\ell\leq X^{\gamma}$ relatively prime with $\ell_0:=\rad(\ell)$, $0<d<k$ with $\gcd(d,k)=1$, and $0<w <\ell_0$ with $\gcd(w,\ell_0)=1$, we have that
  \[
\sum_{\substack{u\leq X \\ u\equiv d\Mod{k} \\ u\equiv \ell w\Mod{\ell_0\ell}}}\sum_{\substack{\nu^2\equiv 3\Mod{\ell_0 u} \\ 0<\nu\leq \ell_0 u}}e\left(h\f{\nu\overline{k}}{\ell_0u}\right)e\left(h'\f{\nu}{\ell_0}\right)\ll X^{\f{5}{6}}
  \]
  and
  \[
\sum_{\substack{u\leq X \\ u\equiv d\Mod{k} \\ u\equiv \ell w\Mod{\ell_0\ell}}}e\left(\pm h\f{\sqrt{u(2X-u)+3}}{k\ell_0u}\right)\sum_{\substack{\nu^2\equiv 3\Mod{\ell_0u} \\ 0<\nu\leq \ell_0u}}e\left(h\f{\nu\overline{k}}{\ell_0u}\right)e\left(h'\f{\nu}{\ell_0}\right)\ll X^{\f{5}{6}},
  \]
where $\overline{k}$ denotes the multiplicative inverse of $k$ modulo $\ell_0u$.
\end{lemma}
\begin{proof}
Since $\gcd(\ell_0 d,k)=1$, there exist integers $w$ and $v$ such that $wk-v\ell_0d=1$. Then $\overline{k}\equiv v\f{\ell_0 u-\ell_0d}{k}+w\Mod{\ell_0 u}$ whenever $u\equiv d\Mod{k}$, so we may rewrite the above two sums as
  \begin{equation}\label{eq6.?1}
\sum_{\substack{u\leq X \\ u\equiv d\Mod{k} \\ u\equiv \ell w\Mod{\ell_0\ell}}}\sum_{\substack{\nu^2\equiv 3\Mod{\ell_0 u} \\ 0<\nu\leq \ell_0 u}}e\left(h\f{\nu\left(v\f{\ell_0u-\ell_0d}{k}+w\right)}{\ell_0u}\right)e\left(h'\f{\nu}{\ell_0}\right)
  \end{equation}
and
\begin{equation}\label{eq6.?2}
\sum_{\substack{u\leq X \\ u\equiv d\Mod{k} \\ u\equiv \ell w\Mod{\ell_0\ell }}}e\left(\pm h\f{(2X-u)^{\f{1}{2}}}{k\ell_0u^{\f{1}{2}}}\right)\sum_{\substack{\nu^2\equiv 3\Mod{\ell_0u} \\ 0<\nu\leq \ell_0u}}e\left(h\f{\nu\left(v\f{\ell_0u-\ell_0d}{k}+w\right)}{\ell_0u}\right)e\left(h'\f{\nu}{\ell_0}\right)
\end{equation}
plus a quantity that is $O(1)$.

Using Lemma~\ref{lem7.1} with $n=\ell_0\ell$ and $m=\f{u}{\ell}$, which are coprime, and using the 1--1 correspondence between solutions $\nu^2\equiv 3\pmod{nm}$ and pairs of solutions $\nu_1^2\equiv 3\pmod{n}$ and $\nu_2^2\equiv 3\pmod{m}$, we can write~\eqref{eq6.?1} and~\eqref{eq6.?2} as the sum over $\ll X^{2\gamma}$ pairs of integers $a,b$ satisfying $\ell_0\ell=a^2-3b^2$, $b\leq\f{a}{2}$, $\gcd(a,b)=1$, and $0<a,b< X^{\gamma}$ of the phase $e\left(h'\left(-\f{\overline{b}_a}{a}-\f{3b}{a(a^2-3b^2)}\right)\right)$ times the quantities
\begin{equation}\label{eq6.?3}
\sum_{\substack{b_1,b_2\Mod{\ell_0 k} \\ b_1^2-3b_2^2\equiv \overline{\ell }d\Mod{k} \\ b_1^2-3b_2^2\equiv w\Mod{\ell_0}}}\sum_{\substack{r^2-3s^2\leq \f{X}{\ell_0 \ell} \\ 0<s\leq\f{r}{2} \\ r\equiv b_1\Mod{\ell_0k}\\ s\equiv b_2\Mod{\ell_0 k } \\ \gcd(ar+3bs,as+br)=1}}\psi_1(ar+3bs,as+br)e\left(-h\f{\overline{as+br}_{ar+3bs}}{ar+3bs}\left(v\ell_0\f{(r^2-3s^2)\ell-d}{k}+w\right)\right)
\end{equation}
and
\begin{equation}\label{eq6.?4}
\sum_{\substack{b_1,b_2\Mod{\ell_0 k} \\ b_1^2-3b_2^2\equiv \overline{\ell }d\Mod{k} \\ b_1^2-3b_2^2\equiv w\Mod{\ell_0}}}\sum_{\substack{r^2-3s^2\leq \f{X}{\ell_0 \ell} \\ 0<s\leq\f{r}{2} \\ r\equiv b_1\Mod{\ell_0k}\\ s\equiv b_2\Mod{\ell_0 k } \\ \gcd(ar+3bs,as+br)=1}}\psi_2(ar+3bs,as+br)e\left(-h\f{\overline{as+br}_{ar+3bs}}{ar+3bs}\left(v\ell_0\f{(r^2-3s^2)\ell-d}{k}+w\right)\right),
\end{equation}
respectively, where
\[
\psi_1(\theta,\mu):=e\left(-h\left(\f{3\mu}{\theta(\theta^2-3\mu^2)}\right)\left(v\f{\theta^2-3\mu^2-\ell_0d}{k}+w\right)\right)
\]
and
\[
\psi_2(\theta,\mu):=e\left(h\left[\pm\f{(2X-\f{\theta^2-3\mu^2}{\ell_0})^{\f{1}{2}}}{k\ell_0^{\f{1}{2}}(\theta^2-3\mu^2)^{\f{1}{2}}}-\left(\f{3\mu}{\theta(\theta^2-3\mu^2)}\right)\left(v\f{\theta^2-3\mu^2-\ell_0d}{k}+w\right)\right]\right).
\]

For each possible value of $c:=ar+3bs$ and each pair $(b_1,b_2)\Mod{\ell_0k}$ satisfying $b_1^2-3b_2^2\equiv \overline{\ell}d\Mod{k}$ and $b_1^2-3b_2^2\equiv w\Mod{\ell_0}$, we will bound the inner sums over the values of $as+br$ in~\eqref{eq6.?3} and~\eqref{eq6.?4}:
\begin{equation}\label{eq6.?5}
  \sum_{\eqref{eq7.???}}\psi_1(c,as+br)e\left(-h\f{\overline{as+br}_c}{c}\left(v\f{c^2-3(as+br)^2-\ell_0d}{k}+w\right)\right)
\end{equation}
and
\begin{equation}\label{eq6.?6}
  \sum_{\eqref{eq7.???}}\psi_2(c,as+br)e\left(-h\f{\overline{as+br}_c}{c}\left(v\f{c^2-3(as+br)^2-\ell_0d}{k}+w\right)\right)
\end{equation}
respectively, where~\eqref{eq7.???} denotes the conditions
\begin{equation}\label{eq7.???}
  \begin{cases}
r^2-3s^2\leq \f{X}{\ell_0\ell}\\0<s\leq\f{r}{2} \\ ar+3bs=c \\ r\equiv b_1\Mod{\ell_0k}\\ s\equiv b_2\Mod{\ell_0k} \\ \gcd(c,as+br)=1
  \end{cases}
\end{equation}
Set
\[
M:=\max_{\eqref{eq7.???}}(as+br)\qquad\text{ and }\qquad m:=\min_{\eqref{eq7.???}}(as+br),
\]
and note that $M\ll c$ and $c^2-M^2\gg c^2X^{-O(\gamma)}$. Using partial summation, we can bound~\eqref{eq6.?5} by
\[
|g_c\left(M\right)||\psi_1\left(c,M+1\right)|+\sum_{\substack{m\leq t\leq M}}|g_c(t)||\psi_1(c,t)-\psi_1(c,t+1)|
\]
and~\eqref{eq6.?6} by
\[
|g_c\left(M\right)||\psi_2\left(c,M+1\right)|+\sum_{\substack{m\leq t\leq M}}|g_c(t)||\psi_2(c,t)-\psi_2(c,t+1)|
,
\]
where
\[
g_c(t):=\sum_{\substack{as+br\leq t \\ \eqref{eq7.???}}}e\left(-h\f{\overline{as+br}_c}{c}\left(v\f{c^2-3(as+br)^2-\ell_0d}{k}+w\right)\right).
\]
Note that $|\psi_1|,|\psi_2|\leq 1$, and
\[
  |\psi_1(c,t)-\psi_1(c,t+1))|\ll \f{X^{O(\gamma)}}{c}
\]
and
\[
  |\psi_2(c,t)-\psi_2(c,t+1))|\ll \f{X^{\f{1}{2}+O(\gamma)}}{c^2}+\f{X^{O(\gamma)}}{c}
\]
when $t\leq M$.

To deduce a bound for $g_c(t)$, we start by writing $c=\ell_0k n_1+c_1$ and $as+br=\ell_0k n_2+c_2$ in the definition of $g_c(t)$ with $c_1:=ab_1+3bb_2$ and $c_2:=ab_2+bb_1$, so that
\[
  g_c(t)=\sum_{\substack{\ell_0kn_2+c_2=as+br\leq t \\\eqref{eq7.???}}}e_{c}\left(-h\overline{(\ell_0kn_2+c_2)}_c\cdot\left(v\left[\ell_0^2k (n_1^2-3n_2^2)+2\ell_0(c_1n_1-3c_2n_2)\right]+w'\right)\right)
\]
for $w'=w+v\f{c_1^2-3c_2^2-\ell_0d}{k}$, where $e_c(z):=e(z/c)$. The above can be rewritten as
\[
  \sum_{\substack{\ell_0kn_2+c_2=as+br\leq t \\\eqref{eq7.???}}}e_{c}\left(-h\cdot\left(-3\ell_0 v n_2+\overline{(\ell_0kn_2+c_2)}_c(-3\ell_0vc_2n_2+c_3(n_1))\right)\right),
\]
where $c_3(n_1)=w'+\ell_0 v c_1n_1$, since $\ell_0 v n_1(\ell_0k n_1+2c_1)\equiv \ell_0 v n_1 c_1\Mod{\ell_0kn_1+c_1}$.

With a view towards using the P\'olya--Vinogradov method to bound $g_c(t)$, we will first consider complete sums of the form
  \[
s(\xi):=\sum_{\substack{t\Mod{c}\\ (c,\ell_0k t+c_2)=1}}e_{c}\left(-h\cdot\left(-3\ell_0 v t+\overline{(\ell_0kt+c_2)}(-3\ell_0vc_2t+c_3(n_1))\right)-\xi t\right)
  \]
  for all $\xi\Mod{c}$. Also define, for all prime powers $p^{e}\| c$, the sum $s_{p^e}(\xi)$ to be
  \[
  \sum_{\substack{t\Mod{p^e}\\ \gcd(p,\ell_0k t+c_2)=1}}e_{p^e}\left(-h\cdot\left(-3\ell_0 v t+\overline{(\ell_0kt+c_2)}\left(\prod_{p^e\neq q^{e'}\| c}\overline{q^{e'}}\right)(-3\ell_0vc_2t+c_3(n_1))\right)-\xi t\right).
    \]
    By the Chinese remainder theorem, we have
    \[
s(\xi)=\prod_{p^e\| c}s_{p^e}(\xi).
    \]
    As a consequence,
\[
|s(\xi)|\ll X^{O(\gamma)} \prod_{\substack{p^e\| c \\ p\nmid hk\ell}}|s_{p^e}(\xi,c_3(n_1))|,
\]
so that, to bound $|s(\xi)|$, it suffices to bound each of the $s_{p^e}(\xi,c_3(n_1))$'s when $p\nmid hk\ell$.

If $p\nmid k\ell$, then $\ell_0k$ is invertible modulo $p^e$ for any $e>0$, and so we can write
\begin{align*}
|s_{p^c}(\xi,c_3(n_1))| &= \left|\sum_{\substack{t\Mod{p^e}\\ \gcd(p,\ell_0 k t+c_2)=1}}e_{p^e}\left(-h\cdot\left(-3\ell_0 v t+\overline{(\ell_0kt+c_2)}Q(-3\ell_0vc_2t+c_3(n_1))\right)-\xi t\right)\right| \\
  &= \left|\sum_{\substack{t\Mod{p^c}\\ \gcd(p,t)=1}}e_{p^c}\left((3h\ell_0 v-\xi)\overline{\ell_0k}t-hQ\left(3\overline{k}vc_2^2+c_3(n_1)\right)\overline{t}\right)\right|
\end{align*}
for $Q=Q_p=\prod_{p^e\neq q^{e'}\| c}\overline{q^{e'}}\not\equiv 0\Mod{p}$. Using that $c_3(n_1)=w'+\ell_0 vc_1n_1$, $w'=w+v\f{c_1^2-3c_2^2-\ell_0d}{k}$, and $wk-v\ell_0d=1$, a short manipulation gives $3\overline{k}vc_2^2+c_3(n_1)\equiv \overline{k}\Mod{p^e}$. Thus,
\[
s_{p^e}(\xi,c_3(n_1))=\left|\sum_{\substack{t\Mod{p^c}\\ (p,t)=1}}e_{p^c}\left((3h\ell_0 v-\xi)\overline{\ell_0k}t-(hQ\overline{k})\overline{t}\right)\right|,
\]
so that $s_{p^e}(\xi,c_3(n_1))$ is a complete Kloosterman sum. We therefore have that
\[
  |s_{p^e}(\xi,c_3(n_1))|\leq(e+1)p^{\f{e}{2}},
\]
since $\gcd(p,h)=1$. (See Theorem~2 of~\cite{Hooley1957}, for example, for a statement of a general bound for Kloosterman sums. The above is the Weil bound for $e=1$, and is due to Sali\'e for $e>1$.) We conclude that
\[
  |s(\xi)|\ll X^{O(\gamma)}\tau(c)c^{\f{1}{2}}.
\]

To use the P\'olya--Vinogradov method, we will also require bounds for the Fourier coefficients of the subset of integers
\begin{align*}
  T_c:=\bigg\{n_2\in\Z:\ \ell_0kn_2+c_2=as+br\leq t,\ 0<r^2-3s^2\leq\f{X}{\ell_0\ell },\ 0<s\leq\f{r}{2},\ ar+3bs=c,& \\
r\equiv b_1\Mod{\ell_0k},\ s\equiv b_2\Mod{\ell_0k},\text{ and }\gcd(c,as+br)=1&\bigg\}
\end{align*}
defined by the conditions $as+br\leq t$ and~\eqref{eq7.???}, modulo $c$. For any $\xi\Mod{c}$, by a tedious but straightforward change of variables and an application of the triangle inequality, we have
\[
|\widehat{T_c}(\xi)|=\left|\sum_{\substack{r'\in I \\ \gcd(3bc,ac-(a^2-3b^2)r')=1}}e\left(\f{\ell_0\ell\xi r'}{3bc}\right)\right|
\]
for some interval $I=I_{a,b,c,b_1,b_2,t,\ell_0,\ell,k}$ of length less than $c$. Thus, we have $|\widehat{T_c}(0)|<c$ and, for $\xi\neq 0$, we have $|\widehat{T_c}(\xi)|\ll X^{O(\gamma)}\|\f{\ell_0\ell\xi}{c}\|\1$.

Noting that $\gcd(\ell_0\ell,c)=1$, we conclude using Parseval's identity that
\[
g_c(t)=\f{1}{c}\sum_{\xi\Mod{c}}s(\xi)\widehat{T_c}(\xi)\ll X^{O(\gamma)}c^{\f{1}{2}+\ve}
\]
Combining this with our bounds involving $\psi_1(c,t)$ and $\psi_2(c,t)$ above, we deduce from our application of partial summation that the sums~\eqref{eq6.?5} and~\eqref{eq6.?6} are $\ll X^{O(\gamma)}c^{\f{1}{2}+\ve}$ and $\ll X^{O(\gamma)}(c^{\f{1}{2}+\ve}+X^{\f{1}{2}}c^{-\f{1}{2}+\ve})$, respectively. Summing over all $\ll X^{\f{1}{2}+O(\gamma)}$ possibilities for $c$, which all satisfy $c\ll X^{\f{1}{2}+O(\gamma)}$, we get that~\eqref{eq6.?5} and~\eqref{eq6.?6} are both $\ll X^{\f{3}{4}+\ve+O(\gamma)}$. Taking $\gamma$ sufficiently small completes the proof of the lemma.
\end{proof}

Now we can finally prove Lemma~\ref{lem5.4}.

\begin{proof}[Proof of Lemma~\ref{lem5.4}]
The proof of Lemma~\ref{lem5.4} follows the same general outline of the proof of Lemma~\ref{lem4.4}. However, we are concerned with a different hyperboloid, so we note it suffices to prove that
  \begin{equation}\label{eq7.??}
    \#\{(x,y,z)\in\Z^3:4x^2+16y^2-z^2=3,\ |z|\leq X,\ y\equiv a\Mod{2\ell},\text{ and }z\equiv b\Mod{8\ell}\}
  \end{equation}
  equals
\[
  C2^{k+m}\prod_{p\mid \ell}\f{1-\f{\chi_{12}(p)}{p}}{1-\f{1}{p^2}}\cdot\f{X}{\ell^2}
  \]
  times
\[
\prod_{i=1}^k(1+\chi_{12}(p_i))^{\delta_{\f{1}{4}}(a\Mod{p_i})}\cdot\prod_{j=1}^{m}\f{(1+\chi_{12}(q_j))^{\delta_{\f{1}{4}}(a\Mod{q_j})}}{2^{\delta_{\pm\f{1}{2}}(a\Mod{q_j})}}\cdot\prod_{p\mid r}\f{(1+\chi_{-1}(p))^{1+\delta_{\f{b}{4}}(a\Mod{p})}}{2^{\delta_{0}(a\Mod{p})}},
\]
plus a quantity that is $O(X^{1-\delta_2})$, for every $a\Mod{2\ell}$ that satisfies $\left(\f{2a+1}{q_j}\right)\left(\f{1-2a}{q_j}\right)\neq -1$ for each $j=1,\dots,m$, for then we can just sum over all $2r\prod_{j=1}^{m}\f{q_j+1}{2}$ of the possible values of $a$ modulo $2rq_1\cdots q_m$.

To estimate~\eqref{eq7.??}, we make the change of variables $u\mapsto z-4y$ and $v\mapsto z+4y$ to write~\eqref{eq7.??} as $4$ times the quantity
\[
\#\{(x,u,v)\in\Z^3:4x^2-uv=3,\ u\in[X],\ v\in[2X-u],\ u\equiv b-4a\Mod{8\ell},\ v\equiv b+4a\Mod{8\ell}\},
\]
which, setting $\ell_0:=\gcd(\ell,b-4a)$ and $k:=\f{8\ell}{\ell_0}$ and using hypotheses~(1),~(2),~(3), and~(4) of the statement of the lemma and our choice of $a\Mod{2\ell}$, we can write as the sum of
\begin{equation}\label{eq7.?}
  2^{k+m}\prod_{j=1}^{m}\f{1}{2^{\delta_{\pm\f{1}{2}}(a\Mod{q_j})}}\cdot\prod_{p\mid r}\f{1+\chi_{-1}(p)}{2^{\delta_{0}(a\Mod{p})}}
\end{equation}
quantities of the form
  \[
2\sum_{\substack{u\leq X \\ u\equiv b-4a\Mod{8\ell}}}\sum_{\substack{4\nu^2\equiv 3\Mod{\ell_0 u} \\ 0<\nu\leq \ell_0 u}}\sum_{\substack{x\equiv \nu'\Mod{\ell_0 u} \\ x\equiv c\Mod{k}}}1_{[2X-u]}\left(\f{4x^2-3}{u}\right)
\]
for
\[
\nu'\equiv \nu+\overline{8\nu}(b+4a)P_u\Mod{\ell_0 u},
\]
where $P_u\equiv p^{e}\Mod{p^{e+1}}$ for each $p^{e}\| u$ with $p\mid \ell$ and $P_u\equiv 0\Mod{p^e}$ for each $p^e\|u$ with $p\nmid 8\ell$, and some $c\Mod{k}$. As in the proof of Lemma~\ref{lem4.4}, the sum above equals
\begin{equation}\label{eq6.5}
\sum_{\substack{u\leq X \\ u\equiv b-4a\Mod{8\ell}}}\sum_{\substack{4\nu^2\equiv 3\Mod{\ell_0u} \\ 0<\nu\leq \ell_0u}}\left[\f{(2X-u)^{\f{1}{2}}}{16\ell u^{\f{1}{2}}}+\psi\left(\f{\sqrt{2uX-u^2+3}-\nu_k}{16\ell u}\right)-\psi\left(\f{-\nu_k}{16\ell u}\right)\right]
\end{equation}
plus a quantity that is $O(\log{X})$, where
\[
\nu_k=\nu'\cdot\overline{k}_{\ell_0 u}k+c\cdot\overline{\ell_0(b-4a)}_{k}\ell_0u.
\]
(Here $\overline{k}_{\ell_0 u}$ denotes the multiplicative inverse of $k$ modulo $\ell_0 u$ and, similarly, $\overline{\ell_0(b-4a)}_{k}$ denotes the multiplicative inverse of $\ell_0(b-4a)$ modulo $k$.)

We first deal with the main term of~\eqref{eq6.5}, which, by Hensel's lemma, equals
\[
\f{1}{16\ell}\sum_{\substack{u\leq X \\ u\equiv b-4a\Mod{8\ell}}}\f{\rho''(\ell_0u)(2X-u)^{\f{1}{2}}}{u^{\f{1}{2}}},
\]
where $\rho''(n):=\#\{\nu\Mod{n}:4\nu^2\equiv 3\Mod{n}\}$. The treatment of this quantity is similar to the treatment of the main term of~\eqref{eq6.1}, except that we will need to derive expressions ourselves for the Dirichlet series $\overline{\chi(\ell_0)}\ell_0^s L_\chi(s)$, where 
\[
  L_{\chi}(s):=\sum_{\ell_0^2\mid n}\f{\rho''(n)\chi(n)}{n^{s}},
\]
for each Dirichlet character $\chi$ modulo $k$. We do this by computing the local factors of these Dirichlet series.

For $p=3$, we have the local factor $1+\f{\chi(3)}{3^s}$, and for all $p>3$ with $p\nmid \ell$, we have the local factor
\[
\left(1-\f{\chi(p)}{p^s}\right)\1 \left(1+\f{(\chi\chi_{12})(p)}{p^s}\right).
\]
For $p\mid k$, the local factor is just $1$, and for $p\mid \ell_0$, we have the local factor
\[
\left(1-\f{\chi(p)}{p^s}\right)\1\f{(1+\chi_{12}(p))\chi(p^2)}{p^{2s}}.
\]
It then follows from a small amount of manipulation that
\[
  L_\chi(s)=\left(1-\f{\chi(3)}{3^s}\right)\rho''(\ell_0)\f{\chi(\ell_0^2)}{\ell_0^{2s}}\prod_{p\mid \ell_0}\left(1+\f{(\chi\chi_{12})(p)}{p^s}\right)\1\f{L(s,\chi)L(s,\chi\chi_{12})}{L(2s,\chi^2)},
\]
so that
\[
\overline{\chi(\ell_0)}\ell_0^sL_\chi(s)=\left(1-\f{\chi(3)}{3^s}\right)\rho''(\ell_0)\f{\chi(\ell_0)}{\ell_0^{s}}\prod_{p\mid \ell_0}\left(1+\f{(\chi\chi_{12})(p)}{p^s}\right)\1\f{L(s,\chi)L(s,\chi\chi_{12})}{L(2s,\chi^2)}.
\]
Thus, $\overline{\chi(\ell_0)}\ell_0^sL_{\chi}(s)$ is holomorphic for $\Re(s)\geq\f{1}{2}$ except, when $\chi$ is the trivial character, for a pole of residue
\[
  C'\f{\rho''(\ell_0)}{\ell_0}\prod_{p\mid \ell_0}\left(1+\f{\chi_{12}(p)}{p}\right)\1\prod_{p\mid k}\f{1-\f{\chi_{12}(p)}{p}}{1+\f{1}{p}}
\]
at $s=1$, where $C'>0$ is an absolute constant. As in the proof of Lemma~\ref{lem4.4}, a standard contour integration tells us that
\[
\sum_{\substack{x\leq t \\ \ell_0\mid x}}\rho''(\ell_0x)\chi(x)\ll (\ell t)^{3/4}
\]
when $\chi$ is nontrivial, and
\[
\sum_{\substack{x\leq t \\ \ell_0\mid x}}\rho''(\ell_0x)\chi(x)=  C'\f{\rho''(\ell_0)}{\ell_0}\prod_{p\mid \ell}\left(1-\f{\chi_{12}(p)}{p}\right)\prod_{p\mid \ell_0}\left(1-\f{1}{p^2}\right)\1\prod_{p\mid k}\left(1+\f{1}{p}\right)\1 t+ O((\ell t)^{\f{3}{4}})
\]
when $\chi$ is trivial. Now, using the orthogonality of Dirichlet characters, we have that
\begin{align*}
\sum_{\substack{x\leq t \\ x\equiv b-4a\Mod{8\ell}}}\rho''(\ell_0x)&=\f{1}{\vp(k)}\sum_{\chi\Mod{k}}\overline{\chi(\ell_0)}\sum_{\substack{x\leq t \\ \ell_0\mid x}}\rho''(\ell_0x)\chi(x) \\
&=C'\f{\rho''(\ell_0)}{\ell_0}\prod_{p\mid \ell}\left(1-\f{\chi_{12}(p)}{p}\right)\prod_{p\mid 8\ell}\left(1-\f{1}{p^2}\right)\1\f{t}{k}+ O((\ell t)^{\f{3}{4}}),
\end{align*}
from which it follows from partial summation that the main term of~\eqref{eq6.5} equals
\[
C''\rho''(\ell_0)\prod_{p\mid \ell}\f{1-\f{\chi_{12}(p)}{p}}{1-\f{1}{p^2}}\cdot\f{X}{(8\ell)^2}+O((\ell X)^{\f{3}{4}}).
\]
Noting that  $p_i\mid\ell_0$ if and only if $a\equiv\f{1}{4}\Mod{p_i}$ and $q_j\mid\ell_0$ if and only if $a\equiv\f{1}{4}\Mod{q_j}$, we get the promised main term in the statement of Lemma~\ref{lem5.4} by multiplying by~\eqref{eq7.?} and summing over the possible choices of $a$ modulo $2rq_1\cdots q_m$.

We must now bound the error term of~\eqref{eq6.5}:
\[
\sum_{\substack{u\leq X \\ u\equiv b-4a\Mod{8\ell}}}\sum_{\substack{4\nu^2\equiv 3\Mod{\ell_0 u} \\ 0<\nu\leq \ell_0 u}}\psi\left(\f{\sqrt{u(2X-u)+3}-\nu_{k}}{16\ell u}\right)-\sum_{\substack{u\leq X \\ u\equiv b-4a\Mod{8\ell}}}\sum_{\substack{4\nu^2\equiv 3\Mod{\ell_0 u} \\ 0<\nu\leq \ell_0u}}\psi\left(\f{-\nu_{k}}{16\ell u}\right).
\]
Let $r_1\dots r_d=\ell_0$ be the prime factorization of $\ell_0$. To deal with the dependence of $P_u$ on $u$, we will split these sums over $u$ up based on the $d$-tuple $(e_1,\dots,e_d)$ for which $r_1^{e_1}\cdots r_d^{e_d}\| u$ and the congruence class of $\f{u}{r_1^{e_1}\cdots r_{d}^{e_d}}$ modulo $\ell_0$. Indeed, for each fixed $d$-tuple $\mathbf{e}=(e_1,\dots,e_d)$ and congruence class $w$ modulo $\ell_0$, there exists a constant $P'_{\mathbf{e},w}$ such that that $P_u=u\cdot P'_{\mathbf{e},w}$ whenever $r_1^{e_1}\cdots r_d^{e_d}\| u$ and  $\f{u}{r_1^{e_1}\cdots r_{d}^{e_d}}\equiv w\Mod{\ell_0}$ (just take $P_u=\overline{w}_{\ell_0}b_{\mathbf{e}}u$ where $\overline{w}_{\ell_0}$ denotes the multiplicative inverse of $w$ modulo $\ell_0$ and $b_{\mathbf{e}}\equiv \prod_{j\neq i}\overline{r_j^{e_j}}\Mod{r_i^{e_i}}$ for each $i=1,\dots,d$). So, with $\gamma$ as in Lemma~\ref{lem7.2}, we write the first sum above as
\[
\sum_{\substack{e_1,\dots,e_d\geq 1 \\ r_1^{e_1}\cdots r_{d}^{e_d}\leq X^{4\delta_2'} \\ w\in(\Z/\ell_0\Z)^\times}}\sum_{\substack{u\leq X \\ u\equiv b-4a\Mod{k} \\ u\equiv r_1^{e_1}\cdots r_d^{e_d}w\Mod{\ell_0 r_1^{e_1}\cdots r_d^{e_d}}}}\sum_{\substack{4\nu^2\equiv 3\Mod{\ell_0u} \\ 0<\nu\leq \ell_0u}}\psi\left(\f{\sqrt{u(2X-u)+3}-\nu_{k}}{16\ell u}\right)+O(X^{1-2\delta_2'+\ve}),
\]
(since any $u$ with $r_1^{e_1}\cdots r_d^{e_d}\| u$ for some $r_1^{e_1}\cdots r_d^{e_d}>X^{4\delta_2'}$ must be divisible by one of at most $(\log{X^{4\delta_2'}})^d$ many integers of size at least $X^{4\delta_2'}$, where $d\leq\f{\log{\ell}}{\log_2{\ell}}\leq 2\delta_2'\f{\log{X}}{\log_2{X}}$, so that $(\log{X^{4\delta_2'}})^{d}\ll X^{2\delta_2'}$) and, similarly, the second sum above as
\[
\sum_{\substack{e_1,\dots,e_d\geq 1 \\ r_1^{e_1}\cdots r_{d}^{e_d}\leq X^{4\delta_2'}\\w\in(\Z/\ell_0\Z)^\times}}\sum_{\substack{u\leq X \\ u\equiv b-4a\Mod{k} \\ u\equiv r_1^{e_1}\cdots r_d^{e_d}w\Mod{\ell_0r_1^{e_1}\cdots r_d^{e_d}}}}\sum_{\substack{4\nu^2\equiv 3\Mod{\ell_0u} \\ 0<\nu\leq \ell_0u}}\psi\left(\f{-\nu_{k}}{16\ell u}\right)+O(X^{1-2\delta'_2+\ve}).
\]

Just like in the proof of Lemma~\ref{lem4.4}, we will insert the Fourier expansion for the sawtooth function to deal with each of the above inner sums over $u$. For every $1\leq M\leq X^{\f{1}{2}}$, we get that the first inner sum over $u$ above can be written as the sum of
\[
\f{1}{\pi}\sum_{h=1}^M\f{1}{h}\cos\left(2\pi h\f{c\cdot\overline{\ell_0(b-4a)}_k}{2k}\right)E_{1}(h),
\]
where $E_1(h)$ equals
\[
  \sum_{\substack{u\leq X \\ u\equiv b-4a\Mod{k} \\ u\equiv r_1^{e_1}\cdots r_d^{e_d}w\Mod{\ell_0r_1^{e_1}\cdots r_d^{e_d}}}}\sin\left(2\pi h\f{\sqrt{u(2X-u)+3}}{16\ell u}\right)\sum_{\substack{4\nu^2\equiv 3\Mod{\ell_0 u} \\ 0<\nu\leq \ell_0 u}}e\left(h\left[\f{\nu\overline{k}_{\ell_0u}}{\ell_0 u}+\f{\nu\overline{24k}_{\ell_0}(b+4a)P_{\mathbf{e},w}'}{\ell_0}\right]\right),
\]
and
\[
  \f{1}{\pi}\sum_{h=1}^M\f{1}{h}\sin\left(2\pi h\f{c\cdot\overline{\ell_0(b-4a)}_{k}}{2k}\right)E_2(h),
\]
where $E_2(h)$ equals
\[
  \sum_{\substack{u\leq X \\ u\equiv b-4a\Mod{k} \\ u\equiv r_1^{e_1}\cdots r_d^{e_d}w\Mod{\ell_0r_1^{e_1}\cdots r_d^{e_d}}}}\cos\left(2\pi h\f{\sqrt{u(2X-u)+3}}{16\ell u}\right)\sum_{\substack{4\nu^2\equiv 3\Mod{\ell_0 u} \\ 0<\nu\leq \ell_0 u}}e\left(h\left[\f{\nu\overline{k}_{\ell_0u}}{\ell_0 u}+\f{\nu\overline{24k}_{\ell_0}(b+4a)P_{\mathbf{e},w}'}{\ell_0}\right]\right),
\]

plus an error term that is at most an absolute constant times
\begin{align*}
  &\sum_{h=1}^{\infty}C_h(M)\cos\left(2\pi h\f{c\cdot\overline{\ell_0(b-4a)}_k}{k}\right)E_2(h)-\sum_{h=1}^{\infty}C_h(M)\sin\left(2\pi h\f{c\cdot\overline{\ell_0(b-4a)}_k}{k}\right)E_1(h)\\
  &+C_0(M)\sum_{\substack{u\leq X \\ u\equiv b-4a\Mod{k} \\ u\equiv r_1^{e_1}\cdots r_d^{e_d}w\Mod{\ell_0r_1^{e_1}\cdots r_d^{e_d}}}}\rho''(u)
\end{align*}
and, similarly, the second inner sum over $u$ above can be written as
\begin{align*}
  \f{1}{\pi}\sum_{h=1}^M\f{1}{h}\sin\left(2\pi h\f{c\cdot\overline{\ell_0(b-4a)}_k}{2k}\right)E_3(h),
\end{align*}
where $E_3(h)$ equals
\[
\sum_{\substack{u\leq X \\ u\equiv b-4a\Mod{k} \\ u\equiv r_1^{e_1}\cdots r_d^{e_d}w\Mod{\ell_0 r_1^{e_1}\cdots r_d^{e_d}}}}\sum_{\substack{4\nu^2\equiv 3\Mod{\ell_0 u} \\ 0<\nu\leq \ell_0 u}}e\left(h\left[\f{\nu\overline{k}_{\ell_0u}}{\ell_0 u}+\f{\nu\overline{24k}_{\ell_0}(b+4a)P_{\mathbf{e},w}'}{\ell_0}\right]\right),
\]

plus an error term that is at most an absolute constant times
\[
\sum_{h=1}^{\infty}C_h(M)\cos\left(2\pi h\f{c\cdot \overline{\ell_0(b-4a)}_k}{2k}\right)E_3(h)+C_0(M)\sum_{\substack{u\leq X \\ u\equiv b-4a\Mod{k} \\ u\equiv r_1^{e_1}\cdots r_d^{e_d}w\Mod{\ell_0 r_1^{e_1}\cdots r_d^{e_d}}}}\rho''(u).
\]

To conclude, we apply Lemma~\ref{lem7.2} to bound each of $E_1(h)$, $E_2(h)$, and $E_3(h)$ by $X^{\f{5}{6}}$ when $hX^{O(\delta_2')}\leq X^{\gamma}$. Combining this with the trivial bound $|E_i(h)|\ll X^{1+\ve}$ for $i=1,2,3$ when $hX^{O(\delta_2')}>X^{\gamma}$, and taking $M=X^{\f{\gamma}{8}}$ and $\delta'_2$ sufficiently small, we get that the above five quantities are $\ll X^{\f{5}{6}+\ve}$, $\ll X^{\f{5}{6}+\ve}$, $\ll X^{1-\f{\gamma}{8}+\ve}+X^{\f{5}{6}+\f{\gamma}{8}+\ve}$, $\ll X^{\f{5}{6}+\ve}$, and $\ll X^{1-\f{\gamma}{8}+\ve}+X^{\f{5}{6}+\f{\gamma}{8}+\ve}$, respectively. This completes the proof of the lemma.
\end{proof}

\bibliographystyle{plain}
\bibliography{bib}

\end{document}